\documentclass[10pt]{article}
\usepackage{authblk}
\usepackage{graphicx} 
\usepackage[utf8]{inputenc}
\usepackage{amsmath}
\usepackage{xcolor}
\usepackage{amssymb}
\usepackage{amsmath}
\usepackage{amsfonts}
\usepackage{amstext}
\usepackage{amsthm}
\usepackage[english]{babel}
\usepackage[autostyle]{csquotes}
\usepackage[super]{nth}
\usepackage{hyperref}
\usepackage[margin=1.1in]{geometry}
\usepackage{txfonts}
\usepackage[normalem]{ulem}
\usepackage{parskip}

\newtheorem{theorem}{Theorem}[section]
\newtheorem{proposition}[theorem]{Proposition}
\newtheorem{lemma}[theorem]{Lemma}
\newtheorem{corollary}[theorem]{Corollary}
\newtheorem{definition}[theorem]{Definition}
\newtheorem{remark}{Remark}

\newtheorem{example}{Example}[section]

{%
\setcounter{enumi}{0}

\begin{enumerate}}%
{\end{enumerate} }

\makeatletter
\newcommand{\myitem}[1]{%
\medskip \item[#1]\protected@edef\@currentlabel{#1}%
} 

\newcommand{\rn}{\mathbb{R}^{n}}
\newcommand{\R}{\mathbb{R}}
\newcommand{\D}{\nabla}
\newcommand{\Fin}[1]{H^{p-1}(\nabla #1)\nabla H(\nabla #1)}
\newcommand{\fin}[1]{H^{p-1}( #1)\nabla H( #1)}

 \numberwithin{equation}{section}

\hypersetup{
    colorlinks=true,
    linkcolor= blue,
    citecolor=blue,
    pdfpagemode=FullScreen,
    }

\title{Boundary Blow-up Solutions for the Finsler p-Laplacian: Wellposedness and Asymptotic Behaviour}

\author{
N N Dattatreya\\
{\small Indian Institute of Technology Kanpur}\\
{\footnotesize \texttt{dattatreya21@iitk.ac.in}}
}
\date{}

\begin{document}


\maketitle
\begin{abstract}
  We study the existence of \textit{large} (\textit{boundary blow-up}) solutions to semilinear equations involving the \textit{Finsler p-Laplacian} on bounded domains with sufficiently smooth boundaries. We establish a \textit{Keller–Osserman}-type condition that ensures the existence of such solutions, and show that this condition retains the same integrability as that of the p-Laplacian. We examine the influence of the anisotropic norm underlying the Finsler p-Laplacian on the boundary behaviour of the solution, then derive asymptotic estimates for large solutions near the boundary of the domain. Using these boundary asymptotics, we prove uniqueness results for power-type nonlinearities. 
\end{abstract}


\section{Introduction}
We investigate non-negative large (boundary blow-up) solutions of semilinear equations driven by the Finsler p-Laplacian. Let $\Omega$ be a bounded domain in $\rn$. We consider positive solutions of the problem
\begin{equation}\label{finsler large eqn}
    \begin{cases}
        \text{div}\left(H^{p-1}(\nabla u)\nabla H(\nabla u)\right)=f(u)\quad &\text{in }\Omega\\
        u(x)\to\infty &\text{as dist}(x,\partial\Omega)\to 0,
    \end{cases}
\end{equation}
here $2\leq p<\infty$ and $H\colon\R^n\rightarrow[0,+\infty)$ has the following properties:
\begin{itemize}
\item $H(x)>0$ for $x\ne 0$.
\item Positive homogeneity of degree 1: $H\left(tx\right)=|t|~H\left(x\right)$ for all $t\in \R$ and $x\in \R^{n}$.
    \item $H\in C^{2}\left(\R^n\setminus \{0\}\right)$.
     \item Strong convexity: The Hessian matrix $\nabla^{2}(H^{2}(x)/2)$ is positive definite for $x\in \rn\setminus\{0\}$.
    \end{itemize}
    
Indeed,  by the second assumption $H(x)=H(-x)$. Further, the assumptions regarding nonlinearity are presented in subsection \ref{assumptions}.

The operator on the left-hand side of \eqref{finsler large eqn} is called the \textit{Finsler p-Laplacian}, denoted by $\Delta_{H}^{p}$. Indeed, at zero, the quantity inside the divergence vanishes conventionally, see \cite{CozziFarinaValdinoci2014}.
 It is an anisotropic generalisation of the Laplace and the p-Laplace operators. In the case $p=2$, the operator $\Delta _{H}^{2}$ is called the Finsler Laplacian, which arises in Finsler geometry; we refer to \cite{BaoChernShen2000} for a comprehensive understanding of this geometry. Apart from geometric origin, anisotropic operators arise in various physical contexts, including minimal surface energy \cite{Taylor1978}, crystallography, crystal growth \cite{TaylorChanHandwerker1992}, image processing \cite{RudinOsherFatemi1992}, and the Willmore functional \cite{Ulrich2004}. The study of these anisotropic problems dates back to 1901, when Wulff \cite{Wulff1901} employed anisotropic geometry to investigate crystals. One of the basic objects of study in crystallography is the \textit{Wulff shape}, which is the unit ball in the dual norm of $H$ and is a minimiser of an anisotropic functional defined on subsets of $\R^{n}$. For more in this direction, we refer to \cite{CozziFarinaValdinoci2014, Ulrich2004} and the references therein. The Wulff shape plays an important role in our analysis.

Large solutions to elliptic equations are studied extensively in various aspects, such as existence, uniqueness, and behaviour (asymptotics) of the solution and its gradient near the boundary of the domain. The necessary and sufficient condition for the existence of large solutions is the Keller-Osserman condition; for the Laplace equation, it reads as follows  
\begin{align*}
    \int^{\infty}_{r}\frac{ds}{\sqrt{F(s)}}<\infty,\quad \text{for all } r>0,
\end{align*}
where $F$ is the primitive of $f$ with $F(0)=0$.

The study of these problems traces back to L. Bieberbach’s 1916 work \cite{biberbach}, but the significant milestones were the independent contributions of Keller \cite{Keller1957} and Osserman \cite{Osserman1957}, both in 1957. The paper \cite{Keller1957} originated in electrodynamics \cite{osti_4157678} while \cite{Osserman1957} originated in geometry. A related development connecting large solutions to stochastic control can be found in \cite{lions}. For the existence, uniqueness, and asymptotics at the boundary of large solutions to semilinear equations, we refer to \cite{BandleMarcus1992, BandleMarcus1995, McKenna, DumontDupaigneOlivierRuadulescu2007} and the references therein. For quasilinear equations involving the Laplace operator, we refer to \cite{BandleAntonioPorru1997, Lieberman}, and for the p-Laplacian, we refer to \cite{DattaIndro2026, diaz1993, materopaper} and the references therein. For the study of large solutions to fully nonlinear equations we refer to \cite{AhmedRuadulescuVitolo2020} and the references therein.
While large solutions for the Laplacian and the p-Laplacian have been studied extensively, much less is known in the anisotropic framework associated with the Finsler operator; see \cite{DellaBlasio2017} for the Hamilton-Jacobi type equation involving the Finsler operator. In the geometric framework, the Finsler Laplacian arises naturally from Finsler geometry. A Finsler metric on a smooth manifold \(M\) is a function \(F(x,y)\) defined on the tangent bundle \(TM\) \cite{BaoChernShen2000,MezeiVas2019}. For each fixed \(x\in M\), the map \(F(x,\cdot)\) induces a Minkowski norm on the tangent space \(T_xM\). In the setting \(M=\Omega\subset\R^n\), identifying \(T_x\Omega\cong\R^n\), the resulting operator is precisely the Finsler-Laplacian.

The present work extends classical results for semilinear equations to an anisotropic framework.

\subsection{Examples of Minkowski norms}
\begin{example}($\ell^{2}-$ norms)
    Consider
$H_{2}(x):=\left(\displaystyle\sum_{i=1}^{n}|x_{i}|^{2}\right)^{\frac{1}{2}}$.
\end{example}
From this choice of $H$, one can obtain the Laplacian operator
       $ div\left(\D u\right)=\Delta u=\Delta_{H_{2}}^{2}u$, and the $p$-Laplacian operator
    $div\left(|\D u|^{p-2}\D u\right)=\Delta_{p}u=\Delta_{H_{2}}^{p}u$.
     \begin{example}
      Let $A$ be an $n\times n$ symmetric positive definite matrix. Take
      $H_{A}(x):=(x^{t}Ax)^{1/2}$.
  \end{example}
  From this, we can get a constant-coefficient elliptic operator of the form $ div\left(A\D u\right)=\Delta_{H_{A}}^{2}$.
  
   Before proceeding to the next example, notice that for $1\leq q<+\infty$, $q\neq 2$, the norms
    \begin{equation*}
H_{q}\left(x\right):=\left(\displaystyle\sum_{i=1}^{n}|x_{i}|^{q}\right)^{\frac{1}{q}}.
   \end{equation*} 
   fail to be in our class since $\nabla^{2}(H_{q}^{2}(x)/2)$ is only positive semi-definite and not positive definite for $2<q<+\infty$ and for $1\leq q<2$ the required smoothness fails, see \cite{BaoChernShen2000,Shenbook2001} for details. Consequently, such norms fall outside the framework considered in this article. Nevertheless, the admissible class of Minkowski norms $H$ satisfying the strong convexity condition remains quite rich.
  
   \begin{example}
       Let $\lambda\geq 0$ and $\mu>0$, consider 
       \begin{equation*}
           H_{\lambda,\mu}(x):=\sqrt{\lambda \sqrt{\sum_{i=1}^{n}|x_{i}|^{4}}+\mu\sum_{i=1}^{n}|x_{i}|^{2}}.
       \end{equation*}
   \end{example}
Clearly, $H_{2}=H_{0,1}$. Moreover, for $\lambda_{1},\lambda_{2},\mu_{1},\mu_{2}>0$ such that $\lambda_{1}/\mu_{1}\neq \lambda_{2}/\mu_{2}$, the norms $H_{\lambda_{1},\mu_{1}}$ and $H_{\lambda_{2},\mu_{2}}$ are non-isomorphic, see \cite{MezeiVas2019}.
\begin{example}
    Let $n=m_{1}+\cdots +m_{r}$ and $1\leq q,\lambda_{i}<+\infty$, $1<p_{i}<+\infty$ for $1\leq i\leq r$, denote $P=\left(p_{1},\cdots,p_{r}\right)$ and $\Lambda=\left(\lambda_{1},\cdots,\lambda_{r}\right)$. Consider 
    \begin{equation*}
        H=H_{q,P,\Lambda}:=\left(\displaystyle \sum_{i=1}^{r}\lambda_{i}\left(\sum_{j=s_{i-1}}^{m_{i}}|x_{j}|^{p_{i}}\right)^{\frac{q}{p_{i}}}\right)^{\frac{1}{q}},
    \end{equation*}
   where $s_{0}=0$ and $s_{k}=\displaystyle\sum_{i=1}^{k}m_{i}$.
\end{example}

\subsection{Assumptions}\label{assumptions}
Throughout the paper, we assume that $p\geq 2$ and that $\Omega$ is a bounded domain in $\R^{n}$ with $C^{2}$ boundary.

We assume the following on the nonlinearity:
\begin{enumerate}
    \myitem {$\textbf{(F)}$}\label{a1}
    \textit{$f:[0,\infty)\to [0,\infty)$ is continuous and strictly increasing, with $f(0)=0$ and $f(r)\to \infty$ as $r\to\infty$.}
\end{enumerate}
In the classical spirit, we assume the Keller-Osserman condition associated with the Finsler p-Laplace equation: Define the primitive of $f$ by
\begin{equation}\label{eqn primitive of f}
    F(x)=\int_{0}^{x}f(s)\,ds.
\end{equation}
Then we assume
\begin{enumerate}
    \myitem{$\textbf{(KO)}$} \label{KO}$$ \Psi_{H,p}(r)=\left(\frac{p-1}{p}\right)^{1/p}\int_{r}^{\infty}\frac{ds}{\left\{F(s)\right\}^{1/p}}<\infty, \quad \textit{for all }~ r>0$$.
\end{enumerate}
\begin{remark}\label{rem KO}
    Remarkably, despite the anisotropic nature of the operator, the Keller–Osserman condition retains the same integrability structure as in the p-Laplacian case. This is because anisotropic norms considered are equivalent to the Euclidean norm. This can also be seen by the reduction to the radial equation where it coincides with that of a standard $p$-Laplacian. See also Remark \ref{remark 1}, and proof of Theorem \ref{existence on annulus} or Proposition \ref{prop annuls asym} in Section \ref{asym}. 
    \end{remark}
    \begin{example}\label{example f=tq}
         If $f(t)=t^{q}$, then \ref{KO} is equivalent to $q>p-1$, and 
    \begin{equation*}
        \Psi_{H,p}(t)= \left(\frac{(p-1)(q+1)}{p}\right)^{1/p}\frac{p}{q+1-p}\frac{1}{t^{\frac{q+1-p}{p}}}.
    \end{equation*}
    \end{example}
We further consider the two possibilities about the integrability of $\Psi_{H,p}$ near $0$:
\begin{enumerate}
        \myitem{$\textbf{(A1)}$}\label{i1} $\int_{0^{+}}\frac{ds}{\{F(s)\}^{1/p}}=\infty$ ($Osgood~ Condition$).
    \myitem{$\textbf{(A2)}$}\label{i2} $\int_{0^{+}}\frac{ds}{\{F(s)\}^{1/p}}<\infty$.
\end{enumerate}

\subsection{Main results}
Consider the following notion of solution.
\begin{definition}\label{dfn solution}
    A function $u\in W_{loc}^{1,p}(\Omega)$ is said to be a solution (subsolution or supersolution)
    \begin{equation}\label{finsler eqn}
         \text{div}\left(H^{p-1}(\nabla u)\nabla H(\nabla u)\right)=f(u)\quad \text{in }\Omega,
    \end{equation}
    if for every pre-compact set $\Omega'\subset\subset \Omega$ 
    \begin{equation*}
        -\int_{\Omega'}\Fin{u}\cdot \D \phi\,dx=(\geq \text{or}\leq)\int_{\Omega'}f(u)\phi\,dx \quad \text{for all }\phi\in W_{0}^{1,p}(\Omega') ~(\phi\geq 0).
    \end{equation*}
    provided $f(u)\in L_{loc}^{p'}(\Omega)$.
\end{definition}
\begin{remark}
    For a fixed $u\in W_{loc}^{1,p}(\Omega)$, and $\Omega'\subset\subset\Omega$, the mapping $v\mapsto \int_{\Omega'}\Fin{u}\cdot \nabla v $ for any $v\in W^{1,p}(\Omega')$ is bounded and linear due to Lemma \ref{lemma_properties of H} below.
\end{remark}

Comparison principle plays an important role in proving both existence and asymptotic results. We prove a comparison principle 

\begin{theorem}\label{comparison principle}
    Let $\Omega$ be a bounded domain in $\rn$ with $C^{2}$ boundary, $p\geq 2$, and $u,v\in W_{loc}^{1,p}(\Omega)\cap C(\Omega)$ be subsolution and supersolution respectively.
  Further, suppose that $u,v \geq 0$, $v>0$ uniformly near the boundary and satisfy 
    \begin{equation}\label{boundary condition for comp princ}
        \limsup_{x\to z}\frac{u(x)}{v(x)}\leq \alpha<1,\quad \text{uniformly for all }z\in \partial\Omega.
    \end{equation}
    Or, if $f(\lambda r)\geq \lambda^{p-1}f(r)$ for all $r\geq 0, \lambda>1$ and
    \begin{equation}\label{boundary condition for comp princ2}
        \limsup_{x\to z}\frac{u(x)}{v(x)}=1,\quad \text{uniformly for all }z\in \partial\Omega.
    \end{equation}
    Then $u\leq v$ in $\Omega$.
    \end{theorem}
 We refer to Definition \ref{dfn solution} for the weak notion. Further, we prove the following existence result
  \begin{theorem}\label{main theorem existence}
     Let $\Omega$ be a bounded domain in $\rn$ with $C^{2}$ boundary. Assume \ref{a1}
 and \ref{KO} then \eqref{finsler large eqn} admits a solution in $W^{1,p}_{loc}(\Omega)$.
 \end{theorem}
 Next, we prove the boundary asymptotic for any large solutions. Before stating the result, the dual norm $H_{0}$ of $H$ is 
\begin{equation}\label{dual dfn}
    H_{0}\left(\xi\right):=\displaystyle\sup_{x\neq 0}\frac{\left<\xi, x\right>}{H\left(x\right)} \quad \text{for all } \xi\in \R^{n}.
\end{equation}
 \begin{theorem}\label{dual asym}
   Let $u$ be a solution of \eqref{finsler large eqn} on a bounded domain $\Omega$ with $C^{2}$ boundary. Let $\delta_{H_{0}}(x):=\min\{H_{0}(x-z)~|~z\in \partial\Omega\}$ 
    Then
       \begin{equation*}
        \lim_{\delta_{H_{0}}(x)\to 0}\frac{\Psi_{H,p}(u(x))}{\delta_{H_{0}}(x)}= 1.
    \end{equation*}   
 \end{theorem}
 Finally, defining $\Phi:=\Psi^{-1}$, we prove the following boundary blow-up result.
  \begin{corollary}\label{thrm asym for f=tq}
        Let $\Omega$ be a domain with $C^{2}$ boundary, and let $u$ be a solution of \eqref{finsler large eqn} with $f(t)=t^{q}$, $q>p-1$. Then 
        \begin{equation}\label{asym f=tq}
            \lim_{\delta_{H_{0}}(x)\to 0}\frac{u(x)}{\Phi(\delta_{H_{0}}(x))}=1.
        \end{equation}
    \end{corollary}
    As a consequence, we obtain the following uniqueness result.
 \begin{theorem}\label{thrm uniqueness}
     Let $\Omega$ be a bounded domain with $C^{2}$ boundary, then the equation \eqref{finsler large eqn} with $f(t)=t^{q}$, $q>p-1$ admits exactly one solution.
 \end{theorem}

\subsection{Outline}

The proof of the existence theorem uses a Browder-Minty-type argument, and relies on a priori bounds on both the solutions and their gradients. Since the symmetry of the solution is naturally with respect to the underlying anisotropic norm, the local bound is obtained using Wulff balls. The proof of the local bound in turn relies on the comparison principle (Theorem \ref{comparison principle}). Although the comparison principle follows a fairly standard argument, we include a complete proof, since, to the best of our knowledge, this particular version does not appear in the existing literature.

The proof of the asymptotic behaviour is motivated by the approach found in \cite{materopaper}. In our setting, the asymptotics depend heavily on the distance defined via the Minkowski norm, the Wulff shape, and the geometry of the domain in Minkowski space. Specifically, one needs to compare the solution with the distance from the boundary measured with respect to the $H_0$ norm. This comparison requires uniform interior and exterior ball conditions associated with the boundary of a $C^2$ domain in Minkowski space; as discussed in the opening paragraphs of Section \ref{asym}. 

The focus of the present work is on extending the theory of large solutions to anisotropic operators of Finsler $p$-Laplacian type. However, we digress by showing the necessity and sufficiency of the Keller-Osserman condition for the existence of a solution in one dimension first, and then use these results to motivate results in higher dimensions; the Keller-Osserman condition is also used to construct a radial solution on a ball with respect to $H_{0}$ norm. This local bound is an essential ingredient in existence in other bounded domains as well as in boundary asymptotics.  While the Keller-Osserman condition retains the same integrability structure as in the $p$-Laplacian case (see Remark \ref{rem KO}), the analysis requires careful handling of the anisotropic geometry induced by $H$ and its dual $H_0$. 

The paper is organised as follows. In Section \ref{prilms}, we collect preliminary results on the Minkowski norm $H$ and its dual $H_0$, and establish the strict monotonicity of the Finsler $p$-Laplace operator. In Section \ref{1d}, we work in one dimension: we derive condition \ref{KO} in this setting, prove Theorem \ref{main theorem existence} for $n = 1$, and establish Theorem \ref{dual asym} for $n = 1$, both by methods that differ from the higher-dimensional argument. In Section \ref{higher dim}, we turn to the higher-dimensional case: we prove the comparison principle (Theorem \ref{comparison principle}), derive the local bounds of Propositions \ref{prop_weak_local_bound} and Proposition \ref{prop_grad lp bound}, and give the proof of Theorem \ref{main theorem existence} for $n \geq 2$. Section \ref{asym} contains the proof of Theorem \ref{dual asym} for $n \geq 2$, and includes a discussion of how the anisotropy inherent in the Finsler operator influences the asymptotic behaviour of the large solutions. In Section \ref{uniqueness}, we apply the asymptotic results to prove Corollary \ref{thrm asym for f=tq} and the uniqueness result, Theorem \ref{thrm uniqueness}. Finally, for completeness, the Appendix presents the existence of solutions to the finite Dirichlet problem via standard minimisation techniques, together with a brief discussion of regularity; as these results are not the main focus of the paper.

\section{Preliminaries}\label{prilms}
This section is dedicated to a discussion of the properties of $H$ and $H_{0}$, as well as the monotonicity of the Finsler p-Laplace operator. We begin by introducing some notation.

 \textbf{Notation:}
Let $\Omega$ be a bounded domain
\begin{itemize}
\item We denote by $H_{0}$ the dual norm of $H$ as defined in \eqref{dual dfn}.
    \item The distance from the boundary of $\Omega$ with respect to the Euclidean norm is $\delta(x):=$ dist$(x,\partial\Omega)$. And with respect to the dual norm $H_{0}$ is $\delta_{H_{0}}(x):=\inf\{H_{0}(x-z)~|~z\in \partial\Omega\}$. Indeed, $H_{0}(x-z)=H_{0}(z-x)$.
    \item We denote $\theta_{1}=\min_{|x|=1}H(x)$ and $\theta_{2}=\max_{|x|=1}H(x)$.
    \item The ball with respect to $H_{0}$ is $\mathcal{W}_{r}(x_{0}):=\{x\in \R^{n}~|~ H_{0}(x-x_{0})<r\}$ for any $x_{0}\in \R^{n}$. In the literature $\mathcal{W}_{1}$ is known as \textit{Wulff shape}. 
    \item For any $\delta>0$, we denote $\Omega_{\delta}:=\{x\in \Omega~|~\delta_{H_{0}}(x)<\delta\}$.
    \item By $A\subset\subset B$, we mean that $A$ is a pre-compact subset of $B$.
\end{itemize}

The following properties hold for $H$ and $H_{0}$.

 \begin{lemma}\label{lemma_properties of H}
    Let $H\in C^{2}(\rn\setminus{0})$ has positivity, positive homogeneity of degree $1$, and strong convexity. Then
    \begin{enumerate}
        \item For $\theta_{1}$ and $\theta_{2}$ as in the Notation, we have $\theta_{1}|z|\leq H\left(z\right)\leq \theta_{2}|z|$ and $0<\theta_{1}\leq \theta_{2}$.
        \item $H\left(x+y\right)\leq H\left(x\right)+H\left(y\right)$.  
\item There exists a constant $c>0$ such that $|\D H|\leq c$.
\item For any vector $x$, $\left<\D H(x), x\right>=H\left(x\right)$.
\end{enumerate}
\end{lemma}
\begin{proof}
    From the homogeneity,  $H\left(z\right)=|z|H\left(\frac{z}{|z|}\right)$, the first one follows from the definitions of $\theta_{1}$ and $\theta_{2}$. Strict inequality $0<\theta_{1}$ follows from the continuity, strict positivity, and homogeneity of $H$. In particular, when $\theta_{1}=\theta_{2}$, $H=\theta_{1}H_{2}$.
The second one holds since strong convexity implies convexity.
    Further, we have by $(2)$ and the homogeneity 
    \begin{equation*}
    \frac{H\left(x+he_{i}\right)-H\left(x\right)}{h}\leq H\left(e_{i}\right). 
    \end{equation*}
    Therefore, the third one thus holds with $c=\left(\displaystyle\sum_{i=1}^{n}H^{2}(e_{i})\right)^{\frac{1}{2}}$.
     Property $(4)$ can be obtained by differentiating $H\left(tx\right)=t H\left(x\right), t>0$, with respect to $t$ and substituting $t=1$.
\end{proof}

\begin{lemma}\label{H0} $H$ and $H_{0}$ interact in the following way
    \begin{enumerate}
     \item $H$-H\"older inequality: $\left<\xi, x\right>\leq H_{0}\left(\xi\right)H\left(x\right)$.
     \item $H_{0}\in C^{2}(\rn\setminus{0})$, $H_{0}(\xi)>0$ for $\xi\neq 0$. Further, it has positive homogeneity of degree $1$, and strong convexity.
        \item Dual of $H_{0}$ is $H$.
        \item $H_{0}\left(\D H(x)\right)=1$ and $H\left(\D H_{0}(x)\right)=1$ for any $x\in\rn\setminus\{0\}$.
        \item $H_{0}(x)\nabla H(\nabla H_{0}(x))=x$ for any $x\in\rn\setminus\{0\}$.
    \end{enumerate}
\end{lemma}
\begin{proof}
    If $x=0$, then $H$-H\"older holds trivially. Otherwise, it holds by 
    \begin{equation*}
        \left<\xi, x\right>=\frac{\left<\xi , x\right>H\left(x\right)}{H\left(x\right)},
    \end{equation*}
    and the definition of $H_{0}$.
    For the proof of the remaining properties, we refer to \cite[Proposition 1.3]{DissertationXia}, \cite[Section 2.1]{BePo1996}, and the references therein.
\end{proof}

Strict monotonicity is one of the key aspects in the proof of the existence of a solution to \eqref{finsler large eqn}. We dedicate the rest of the results to proving that the operator we have considered is strictly monotonic.
\begin{lemma}\label{euler inequality}
    Let $x,y\in \rn\setminus\{0\}$. Then we have 
    \begin{equation}\label{convixity and tangent plane}
        H(x)\geq \left<\D H(y), x\right>.
    \end{equation}
\end{lemma}
\begin{proof}
     Strong convexity of $H$ implies convexity. Thus, for any $t\in (0,1)$,
    \begin{equation*}
        H(x)-H(y)\geq \frac{H(y+t(x-y))-H(y)}{t}.
    \end{equation*}
       As directional derivatives exist, by taking $t\to 0$, we obtain  
     \begin{equation}\label{eqn hyperplane}
         H(x)-H(y)\geq \left<\D H(y), x-y\right>.
     \end{equation}
     Using the last implication of Lemma \ref{lemma_properties of H} on $H(y)$, we obtain  \eqref{convixity and tangent plane}.
\end{proof}
\begin{lemma}\label{lemma monotone operator}
    For any $x,y \in \rn$
    \begin{equation*}
        \left< \fin{x}-\fin{y}, x-y\right>\geq 0.
    \end{equation*}
\end{lemma}
\begin{proof}
    If $y=0$, then by Lemma \ref{lemma_properties of H} the inequality holds in view of the positivity of $H$. The case $x=0$ is handled similarly. Therefore, we may assume that $x\neq 0$ and $y\neq 0$. By $(4)$ in Lemma \ref{lemma_properties of H} and Lemma \ref{euler inequality} 
\begin{align*}
    &\left<\fin{x}-\fin{y}, x-y\right>\\
    &= H^{p}(x)-\left<\fin{x},  y\right>
    +H^{p}(y)-\left<\fin{y}, x\right>\\
    &\geq H^{p}(x)-H^{p-1}(x) H(y)
    +H^{p}(y)-H^{p-1}(y) H(x)\\
    &\geq \left(H(x)-H(y)\right)\left(H^{p-1}(x)-H^{p-1}(y)\right)\geq 0. \qedhere
\end{align*}
 \end{proof}
 Notice that Lemma \ref{lemma monotone operator} holds if $H$ is only assumed to be convex. Strict convexity implies strict monotonicity, as implied in the next three results. We first consider an improvement of Lemma \ref{euler inequality}, which can be found in \cite[Proposition 4.6]{Ohta2009}.
\begin{lemma}
    There exists $C\geq 1$ such that for all $x,y\in\rn $ 
    \begin{equation*}
        \frac{H^{2}(x)+H^{2}(y)}{2}\geq H^{2}\left(\frac{x+y}{2}\right)+\frac{1}{4C^{2}}H^{2}(x-y).
    \end{equation*}
\end{lemma}
The following two lemmas are from \cite{BalGarainMukarjee2021}.
\begin{lemma}
    Let $p\geq 2$, there exists a constant $c>0$ such that for all $x,y\in\rn$
    \begin{equation}\label{eqn improved hyperplane}
        H^{p}(x)\geq H^{p}(y)+\left<\nabla H^{p}(y), x-y\right>+cH^{p}(x-y).
    \end{equation}
\end{lemma}
\begin{proof}
    Since $p\geq 2$, we use the inequality $a^{p}+b^{p}\leq (a^{2}+b^{2})^{p/2}$ for $a,b\geq 0$, homogeneity of $H$, previous lemma and $(a+b)^{p}\leq 2^{p-1}(a^{p}+b^{p})$ to obtain
    \begin{equation*}
    \begin{split}
       H^{p}\left(\frac{x+y}{2}\right)+\frac{1}{C^{p}}H^{p}\left(\frac{x-y}{2}\right)&\leq \left\{H^{2}\left(\frac{x+y}{2}\right)+\frac{1}{C^{2}}H^{2}\left(\frac{x-y}{2}\right)\right\}^{\frac{p}{2}} \\
       &= \left\{H^{2}\left(\frac{x+y}{2}\right)+\frac{1}{4C^{2}}H^{2}\left(x-y\right)\right\}^{\frac{p}{2}}\\
       &\leq \frac{1}{2^{\frac{p}{2}}}\left\{H^{2}(x)+H^{2}(y)\right\}^{\frac{p}{2}}\\
       &\leq \frac{H^{p}(x)+H^{p}(y)}{2}.
    \end{split}
    \end{equation*}
    Since $H^{p}$ convex, same logic as in Lemma \ref{euler inequality} and replacing $x$ by $(x+y)/2$ in \eqref{eqn hyperplane} with $H^{p}$ instead of $H$, and using it in the above inequality we obtain 
    \begin{equation*}
        H^{p}(y)+\frac{1}{2}\left<\nabla H^{p}(y), x-y\right>+\frac{1}{C^{p}}H^{p}\left(\frac{x-y}{2}\right)\leq \frac{H^{p}(x)+H^{p}(y)}{2}.
    \end{equation*}
    The lemma follows from the homogeneity, with $c=1/2C^{p}$.
\end{proof}
\begin{corollary}\label{strict convexity of H}
    Let $p\geq 2$. There exists a positive constant $c=c(H,p)$ such that for any $x,y\in \rn$
    \begin{equation*}
          \left<\fin{x}-\fin{y}, x-y\right>\geq c H^{p}(x-y).
    \end{equation*}
    \end{corollary}
\begin{proof}
Interchanging $x$ and $y$ in \eqref{eqn improved hyperplane} and adding the resulting inequality to \eqref{eqn improved hyperplane}, the corollary follows.
\end{proof}

\section{Analysis in one dimension}\label{1d}
In this section, we restrict ourselves to one dimension and show that the Keller–Osserman condition is necessary and sufficient for the existence of large solutions on an interval, consequently proving Theorem \ref{main theorem existence} for $n=1$. Further, we explore the behaviour of a solution near the endpoints of the interval and subsequently prove Theorem \ref{dual asym} for $n=1$.

The following subsection follows an approach similar to that presented in \cite{BandleChipot2013,DattaIndro2026} for intervals centred at the origin. However, we use a translation argument to extend the existence result for general intervals.
\subsection{Existence:}
Let $a,b\in \R$. We consider the problem 
\begin{equation}\label{eqn in one dim without boundary}
    \left(H^{p-1}(u')H'(u')\right)'=f(u)\quad \text{on } (a,b).
\end{equation}
In dimension one,being positively 1-homogeneous, must be of the form $H(t)=\gamma |t|$ for some $\gamma>0$. The above equation reduces to  
\begin{equation}\label{eqn reduced in one dim}
    \gamma^{p}\left(|u'|^{p-2}u'\right)'=f(u)\quad \text{on }(a,b).
\end{equation}
If $u$ is a classical solution,
since $f\geq 0$ by \ref{a1}, the equation implies $(p-1)\gamma^{p}|u'|^{p-2}u''\geq 0$. That is,  $u$ must be convex. Therefore, we look for a convex solution.

    With this at hand, we prove Theorem \ref{main theorem existence} for $n=1$.
    \begin{theorem}\label{thrm existence in one dim with osgood}
        Assume \ref{KO}, \ref{a1}, and \ref{i1}. Then  
        \begin{equation}\label{eqn in one dim}
        \begin{cases}
      \left(H^{p-1}(v')H'(v')\right)'=f(v)\quad \text{on } (a,b) \\
      v(x)\to \infty \quad \text{as } x\to a \text{ or }b,
        \end{cases}
        \end{equation}
        admits a unique solution.
    \end{theorem}
    \begin{proof}
    We'll try to construct a convex solution, for this, we assume an initial value; let $c_{m}$ be the point at which solution should attain minimum. Let $v_{0}$ is the minimum value. We prove the existence in three steps; in the first step, we associate a function $v_{0} \mapsto \ell(v_{0})$ which determines an interval $\left(-\ell(v_{0}),\ell(v_{0})\right)$ on which a unique solution exists. In the next step we show that the map $\ell:(0,\infty)\to (0,\infty)$ is a bijection, which implies that given any $\ell>0$, we can construct the solution on $(-\ell,\ell)$. Finally, in the last step we use translations to get unique solution on any interval $(a,b)$.

    \textit{Step 1:} We consider the second-order initial value problem
\begin{equation}\label{Ode1}
 \begin{cases}
     \gamma^{p}\left(|v'|^{p-2}v'\right)'=f(v)\quad \text{for all } x\geq c_{m}\\
     v'(c_{m})=0 \quad \text{and} \quad v(c_{m})=v_{0}.
 \end{cases}
 \end{equation}
 Multiplying the above equation by $v'$, integrating from $c_{m}$ to $x$, and using the change of variables $s=v(t)$, we obtain 
 \begin{equation*}
     (p-1)\gamma^{p}\int_{0}^{v'(x)}t^{p-1}\,dt=F\left(v(x)\right)-F(v_{0})
 \end{equation*}
 where $F$ is defined in \eqref{eqn primitive of f}.
Rewriting the above equation
 \begin{equation}\label{eqn quotient}
     \left(\frac{p-1}{p}\right)^{1/p}\frac{\gamma v'(x)}{\left\{F\left(v(x)\right)-F(v_{0})\right\}^{1/p}}=1.
 \end{equation}
 Integrating from $c_{m}$ to $x$ and using a change of variables,
 \begin{equation*}
     \left(\frac{p-1}{p}\right)^{1/p}\gamma \int_{v_{0}}^{v(x)}\frac{ds}{\left\{F(s)-F(v_{0})\right\}^{1/p}}=x-c_{m}.
 \end{equation*}
 \begin{remark}
     If $v(x)\to +\infty$ as $x\to b$, then by the by the preceding identity \ref{KO} holds. On the other hand if \ref{KO} holds, then $v$ has to blow up at a point in $\R$ since otherwise right hand side can tend to $+\infty$.
 \end{remark}
  \begin{remark}\label{remark 1}
     The appearance of anisotropy in the Keller-Osserman condition occurs only through the multiplicative constant $\gamma$; consequently, the integrability requirement is invariant under the choice of anisotropy $H$.
 \end{remark}
Uniqueness of $v$ follows by the positivity of the integrand. 

Let $(c_{m},\ell(v_{0}))$ be the maximal interval of existence. Integrating \eqref{eqn quotient} from $c_{m}$ to $\ell(v_{0})$, we obtain  
 \begin{equation}\label{eqn dfn of ell}
    \ell(v_{0})=c_{m}+ \gamma \left(\frac{p-1}{p}\right)^{1/p}\int_{v_{0}}^{\infty}\frac{ds}{\{F(s)-F(v_{0})\}^{1/p}}.
\end{equation}
The quantity $\ell(v_{0})<+\infty$ due to \ref{KO}.
Similarly, let $\Tilde{v}$ be the unique solution of
\begin{equation}\label{ODE2}
   \begin{cases}
     \gamma^{p}\left(|\Tilde{v}'|^{p-2}\Tilde{v}'\right)'=f(\Tilde{v})\quad \text{for all } x\leq c_{m}\\
     \Tilde{v}'(c_{m})=0 \quad \text{and} \quad \Tilde{v}(c_{m})=v_{0}.
 \end{cases} 
\end{equation}
It is given by 
\begin{equation}\label{eqn u2}
    \gamma \left(\frac{p-1}{p}\right)^{1/p}\int_{v_{0}}^{\Tilde{v}(x)}\frac{ds}{\{F(s)-F(v_{0})\}^{1/p}}=c_{m}-x\quad \text{on } (\Tilde{\ell}(v_{0}),c_{m}).
\end{equation}
Define the $C^{1}(\tilde{\ell}(v_{0}),\ell(v_{0}))$ function $w$, where $w(x):=v(x)$ for $x\geq c_{m}$ and $w:=\tilde{v}(x)$ for $x\leq c_{m}$. Then $w$ is the unique solution of \eqref{eqn in one dim without boundary} when $a=\Tilde{\ell}(v_{0})$ and $b =\ell(v_{0})$. By the observation from the beginning of this section, $w$ is convex. In particular, $v_{0}$ is the minimum of $w$.
By uniqueness of the solution $w(x)=w(2c_{m}-x)$, which implies $c_{m}=\frac{\ell(v_{0})+\Tilde{\ell}(v_{0})}{2}$. Thus, up to translations, it is enough to assume that $c_{m}=0$, and consider the interval $(-\ell(v_{0}),\ell(v_{0}))$.

  \textit{Step 2:} The mapping $\ell:(0,\infty)\to(0,\infty)$ is bijective.

 Applying change of variables to \eqref{eqn dfn of ell}, we obtain 
        \begin{equation*}
            \ell(t)=\gamma \left(\frac{p-1}{p}\right)^{1/p}\int_{0}^{\infty}\frac{ds}{\left\{F(s+t)-F(t)\right\}^{1/p}}.
        \end{equation*}
First, $\ell$ is continuous by the dominated convergence theorem.

For $t_{1}<t_{2}$,
\begin{equation*}
    F(s+t_{2})-F(t_{2})=\int_{0}^{s}f(t_{2}+r)\,dr>\int_{0}^{s}f(t_{1}+r)\,dr=F(s+t_{1})-F(t_{1}).
\end{equation*}
This implies $\ell(t_{1})>\ell(t_{2})$. 
Hence, $\ell$ is injective.

To show that $\ell$ is onto, it is sufficient to show that $\ell(t)\to \infty$ as $t\to 0$ and $\ell(t)\to 0$ as $t\to \infty$.

        Since $F(s+t)-F(t)\to F(s)$ as $t\to 0$, the Fatou's lemma together with Osgood condition \ref{i1} implies that $\ell(t)\to \infty$ as $t\to 0$. 
        
        For the second limit, as $f$ is increasing, we can write
        \begin{equation*}
           F(s+t)-F(t)=\int_{t}^{s+t}f(r)\,dr\geq sf(t).
        \end{equation*}
    
       Thus 
       \begin{equation*}
           \left\{F(s+t)-F(t)\right\}^{1/p}\geq (sf(t))^{1/p}\to +\infty \text{ as } t\to +\infty.
       \end{equation*}
      Indeed, if $f\leq M$ for some $M\in (0,+\infty)$, we would have $F(t)^{-1/p}\geq (Mt)^{-1/p}$, and so \ref{KO} fails. Thus, $\ell(t)\to 0$ as $t\to \infty$ by dominated convergence theorem and \ref{KO}.

  \textit{Step 3:}
        Let $c=\frac{a+b}{2}$ and $\delta=\frac{b-a}{2}$. By the previous step, there exists $t_{0}\in (0,\infty)$ such that $\ell(t_{0})=\delta$. The function $\phi$ defined implicitly by 
        \begin{equation*}
              \gamma\left(\frac{p-1}{p}\right)^{1/p}\int_{t_{0}}^{\phi(x)}\frac{ds}{\{F(s)-F(t_{0})\}^{1/p}}=|x|, \quad \text{for all } x\in (-\delta,\delta),
        \end{equation*}
         solves \eqref{eqn in one dim} for $a=-\delta$ and $b=\delta$. Thus, $v(x)=\phi(x-c)$ is the desired solution.
\end{proof}

 With the assumptions present in the previous theorem, it is natural to ask if  $\ell(v_{0})=+\infty$, or equivalently, $v_{0}=0$. The answer is negative. In this regard, the next two results provide the non-existence of a large solution on $\R$.
\begin{lemma}
    Assume \ref{i1}. Let $u$ be the solution of \eqref{eqn in one dim}. Then $u_{0}:=\min_{x}u(x)> 0$.
\end{lemma}
\begin{proof}
    As $u\geq 0$, $u_{0}\geq 0$. If $u_{0}=0$, let $t_{0}:=\min\{t~|~u(t)>0\}$. Then, $t_{0}\geq c_{m}$. For any $t_{0}<x<y$, integrating \eqref{eqn quotient} between $x$ and $y$ implies
    \begin{equation*}
        \gamma \left(\frac{p-1}{p}\right)^{1/p}\int_{u(x)}^{u(y)}\frac{ds}{\left\{(F(s)\right\}^{1/p}}=y-x.
    \end{equation*}
    As $x\to t_{0}^{+}$, the monotone convergence theorem and \ref{i1} gives a contradiction. 
    \end{proof}
    \begin{corollary}
        Let $f$ satisfy \ref{KO} and \ref{i1}. There exists no function $u$ solving \eqref{eqn in one dim} with $a=-\infty$ and $b=\infty$. 
    \end{corollary}
    \begin{proof}
     If such a solution $u$ exists, then by \eqref{eqn dfn of ell} and \ref{KO} $u_{0}(:=\displaystyle\min_{\R}u(x))=0$, a contradiction to the above lemma.
    \end{proof}

  \begin{remark}
    In the proof of step 2 of Theorem \ref{thrm existence in one dim with osgood}, the Osgood condition is not used to prove $\ell(t)\to 0$ as $t\to \infty$.
 \end{remark}
 Now we come to the second possibility \ref{i2}. Define 
 \begin{equation*}
     L=\gamma \left(\frac{p-1}{p}\right)^{1/p}\int_{0}^{\infty} \frac{ds}{\left\{F(s)\right\}^{1/p}}.
 \end{equation*}
 In view of the proofs of Theorem \ref{thrm existence in one dim with osgood}, one has
 \begin{theorem}\label{thrm existence without osgood}
     Assume \ref{KO}, \ref{a1}, and \ref{i2}. For any $a<b$ in $\R$ with $b-a\leq 2L$, there exists a unique $v$ solving \eqref{eqn in one dim}.
 \end{theorem} 
 \begin{proof}
     The proof follows the same approach as Theorem \ref{thrm existence in one dim with osgood}; the only change is that when $t\to 0$, $\ell(t)\to L$ by  the dominated convergence theorem.
 \end{proof}
 If $b-a>2L$, a solution can be constructed. To this end, up to a translation, we assume that $a=-b$. Let $v_{L}$
be the solution of \eqref{eqn in one dim} for $-a=b=L$ given by Theorem \ref{thrm existence without osgood}. Define the unique solution in this case as follows
\begin{align*}
      \Tilde{v}(x)=0 \quad \forall~x\in [L-b,b-L].
  \end{align*}
 And it is implicitly defined by
 \begin{align*}
     \gamma\left(\frac{p-1}{p}\right)^{1/p}\int_{0}^{\Tilde{v}(x)}\frac{ds}{\left\{F(s)\right\}^{1/p}}=|x-b+L| \quad \forall x\in (-b, L-b)\cup (b-L, b).
 \end{align*}
\begin{theorem}
      Assume \ref{KO}, \ref{a1}, and \ref{i2}. For any $a<b$ in $\R$ with $b-a> 2L$, there exists a unique $v$ solving \eqref{eqn in one dim}. Also, $v=0$ in $[L-b+c,b-L+c]$.
\end{theorem}
\begin{proof}
    Aided by the above discussion, the solution is given by $v(x)=\Tilde{v}(x-c)$.
\end{proof}
\subsection{Boundary asymptotics:}

\begin{proposition}
    Let $u$ be a solution of \eqref{eqn in one dim}. Then
    \begin{equation*}
       \lim_{x\to b}\frac{\Psi_{H,p}(u(x))}{b-x}= \frac{1}{\gamma},
    \end{equation*}
 and 
  \begin{equation*}
       \lim_{x\to a}\frac{\Psi_{H,p}(u(x))}{x-a}= \frac{1}{\gamma}.
    \end{equation*}
\end{proposition}
\begin{proof}
    Clearly 
    \begin{equation*}
        \left\{F(s)\right\}^{1/p}\geq \left\{F(s)-F(v_{0})\right\}^{1/p}.
    \end{equation*}
   On the other hand, since $F(s)\to \infty$ as $s\to\infty$, for a fixed $\eta>1$, there exists $s(\eta)>0$ such that $F(s)\geq \eta F(v_{0})$ for all $s\geq s(\eta)$;
   which implies
   \begin{equation*}
       \left\{F(s)-F(v_{0})\right\}^{1/p}\geq \left(1-\frac{1}{\eta}\right)^{1/p}\left\{F(s)\right\}^{1/p}.
   \end{equation*}
   With these, one can write
   \begin{equation}\label{eqn double inequality}
       \left(1-\frac{1}{\eta}\right)^{1/p}\frac{1}{ \left\{F(s)-F(v_{0})\right\}^{1/p}}\leq \frac{1}{\left\{F(s)\right\}^{1/p}}\leq \frac{1}{ \left\{F(s)-F(v_{0})\right\}^{1/p}}.
   \end{equation}
   Since $u(t)\to\infty$ as $t\to b$, we obtain $r>0$ such that for all $t>b-r$, $u(t)\geq s(\eta)$. For $t>b-r$ integrating the above inequality from $u(t)$ to $\infty$, we obtain 
   \begin{equation}\label{eqn double integral ineq}
       \gamma\left(\frac{p-1}{p}\right)^{1/p}\left(1-\frac{1}{\eta}\right)^{1/p}\int_{u(t)}^{\infty}\frac{ds}{ \left\{F(s)-F(v_{0})\right\}^{1/p}}\leq \gamma\Psi_{H,p}(u(t)) \leq \gamma\left(\frac{p-1}{p}\right)^{1/p}\int_{u(t)}^{\infty}\frac{ds}{ \left\{F(s)-F(v_{0})\right\}^{1/p}}.
   \end{equation}
   Integrating \eqref{eqn quotient} from $t$ to $b$,  employing  the change of variables $s=v(x)$, and using this in the above inequality we obtain 
   \begin{equation*}
       \left(1-\frac{1}{\eta}\right)^{1/p}\leq \frac{\gamma\Psi_{H,p}(u(t))}{b-t}\leq 1.
   \end{equation*}
   As $t\to b$, therefore we obtain  
   \begin{equation*}
       \left(1-\frac{1}{\eta}\right)^{1/p}\leq \liminf_{t\to b}\frac{\gamma\Psi_{H,p}(u(t))}{b-t}\leq \limsup_{t\to b}\frac{\gamma\Psi_{H,p}(u(t))}{b-t}\leq 1.
   \end{equation*}
   and also $\eta\to \infty$, we obtain 
    \begin{equation*}
     \lim_{t\to b}\frac{\Psi_{H,p}(u(t))}{b-t}= \frac{1}{\gamma}.
   \end{equation*}
   
   Next, as $u(t)\to \infty$ as $t\to a$, there exists $r>0$ such that for all $t<a+r$ we have $u(t)>s(\eta)$. The counterpart of \eqref{eqn quotient} for $u=\Tilde{v}$, which solves \eqref{ODE2} is 
   \begin{equation*}
       \gamma\left(\frac{p-1}{p}\right)^{1/p}\frac{-u'(x)}{\left\{F\left(u(x)\right)-F(v_{0})\right\}^{1/p}}=1.
   \end{equation*}
   Integrating the above equation from $a$ to $t$ and using a change of variables $s=v(x)$, we obtain  
   \begin{equation*}
       t-a=\gamma\left(\frac{p-1}{p}\right)^{1/p}\int_{u(t)}^{\infty}\frac{ds}{\left\{F(s)-F(v_{0})\right\}^{1/p}}.
   \end{equation*}
   Using this equation in \eqref{eqn double integral ineq}, we have 
    \begin{equation*}
       \left(1-\frac{1}{\eta}\right)^{1/p}\leq \frac{\gamma\Psi_{H,p}(u(t))}{t-a}\leq 1.
   \end{equation*}
   As $\eta\to \infty$ we obtain 
   \begin{equation*}
       \lim_{t\to a}\frac{\Psi_{H,p}(u(t))}{t-a}=\frac{1}{\gamma}.\qedhere
   \end{equation*}
\end{proof}
Since $\delta(x)=min\{b-x,x-a\}$ we have the following boundary asymptotics
\begin{theorem}\label{theorem boy asylum 1d}
        Let $u$ be a solution of \eqref{eqn in one dim}. Then
    \begin{equation*}
       \lim_{\delta(x)\to 0}\frac{\Psi_{H,p}(u(x))}{\delta(x)}= \frac{1}{\gamma}.
    \end{equation*}
\end{theorem}

    If we consider the integration with respect to the measure corresponding to $H_{0}$, we can have  
   \begin{proof}[\textbf{Proof of Theorem \ref{dual asym} for $n=1$}]
      Since $H_{0}(x)=\frac{|x|}{\gamma}$, we have $\delta(x)=\gamma \delta_{H_{0}}(x)$. Thus, Theorem \ref{theorem boy asylum 1d} proves Theorem \ref{dual asym}.
   \end{proof}

\section{Existence in higher dimension}\label{higher dim}
 In this section, we discuss the existence of a solution to \eqref{finsler large eqn} in higher dimensions.  We begin with the proof of the comparison principle stated in Theorem \ref{comparison principle}. 

    \begin{proof}[\textbf{Proof of Theorem \ref{comparison principle}}]
    %
     
By \eqref{boundary condition for comp princ} we have $\delta>0$ be such that $u-v< 0$ in $\Omega_{\delta}:=\{x\in\Omega~|~ d(x,\partial\Omega)<\delta\}$. 
Take supp$(u-v)_{+}:=\Omega'\subset \Omega\setminus\Omega_{\delta}$.

On the one hand, since $f$ is strictly increasing, we have
\begin{equation}\label{controduction ineq1}
    \int_{\Omega'}\left(f(u)-f(v)\right)(u-v)>0.
\end{equation}
On the other hand, by the definition of sub- and super-solutions we have for all $\phi\in W_{0}^{1,p}(\Omega')$ with $\phi\geq 0$ that
\begin{equation}\label{eqn weak comp ineq}
     \int_{\Omega'}\left<\Fin{u}-\Fin{v}, \nabla \phi\right>+(f(u)-f(v))\phi \,dx\leq 0.
\end{equation}
 Since $(u-v)_{+}\in W_{0}^{1,p}(\Omega')$, 
 $\phi=(u-v)_{+}$ is an admissible text function, we see by Lemma \ref{lemma monotone operator} (with $x=\nabla u$ and $y=\nabla v$), that the first integral is non-negative. Thus
\begin{equation*}
    \int_{\Omega'} \left(f(u)-f(v)\right)(u-v)\,dx\leq 0
\end{equation*}


A contradiction to \eqref{controduction ineq1}.  Thus, $u\leq v$ on $\Omega$. 

Next we assume $f(\lambda r)\geq \lambda^{p-1}f(r)$ for $r\geq 0, \lambda>1$ and \eqref{boundary condition for comp princ2}. If $u\leq \lambda v$ on $\Omega$ for all $\lambda>1$, then $u\leq v$ on $\Omega$. Assume that there exists $\lambda>1$ such that $u>\lambda v$ on some subset of $\Omega$. Fix such a $\lambda$ and define $w=\lambda v$. Then, for any $\Omega'\subset\subset\Omega$, $\phi\in W_{0}^{1,p}(\Omega')$ with $\phi\geq 0$, one has
\begin{equation*}
        -\int_{\Omega'}\left<\Fin{w}, \nabla \phi\right>\,dx= -\lambda^{p-1}\int_{\Omega'}\left<\Fin{v}, \nabla \phi\right>\,dx\leq \int_{\Omega'}\lambda^{p-1}f(v)\phi\,dx\leq \int_{\Omega'}f(w)\phi\,dx.
  \end{equation*}
  Also,
  \begin{equation*}
      \limsup_{x\to z}\frac{u(x)}{w(x)}=\frac{1}{\lambda}\limsup_{x\to z}\frac{u(x)}{v(x)}=\frac{1}{\lambda}<1 \quad \text{uniformly for all }z\in \partial\Omega.
  \end{equation*}
  Thus, by the previous case $u\leq \lambda v$ on $\Omega$, a contradiction to the choice of $\lambda$. Thus $u\leq \lambda v$ for all $\lambda>1$ and thus $u\leq v$.
  \end{proof}

We derive a $W^{1,p}_{loc}$ a priory bound for a solution.

 \begin{proposition}\label{prop_grad lp bound}
     Let $\Omega$ be a bounded domain and let $u\in W_{loc}^{1,p}(\Omega)\cap C(\Omega)$ be the solution of \eqref{finsler large eqn} then there exists a positive constant $C=C(p,K, \theta_{1},\theta_{2})$ for any compact set $K$ in $\Omega$ such that  
     \begin{equation*}
         \|\nabla u\|^{p}_{L^{p}(K)}\leq C \left(1+f\left(\|u\|_{L^{\infty}(\Tilde{K})}\right)\right)\|u\|_{L^{\infty}(\Tilde{K})}.
     \end{equation*}
     for some compact set $\Tilde{K}$ such that $K\subset \Tilde{K}\subset \Omega$.
 \end{proposition}
 \begin{proof}
     Let $\phi\in C_{c}^{\infty}(\Tilde{K})$, $0\leq\phi\leq 1$ and $\phi=1$ on $K$. Substituting $u\phi^{p}$ in the weak formulation and using Lemma \ref{lemma_properties of H}, we obtain
     \begin{equation*}
         \begin{split}
             \theta_{1}^{p}\int_{\Tilde{K}}|\nabla u|^{p}\phi^{p} &\leq \int_{\Tilde{K}}H^{p}(\nabla u) \phi^{p}\,dx=-\int_{\Tilde{K}}f(u)u\phi^{p}\,dx-p\int_{\Tilde{K}}u\phi^{p-1}\left<\Fin{u}, \nabla \phi\right>\,dx\\
             &\leq f(\|u\|_{L^{\infty}(\Tilde{K})})\|u\|_{L^{\infty}(\Tilde{K})}|\Tilde{K}|+p\theta_{2}^{p-1}\|u\|_{L^{\infty}(\Tilde{K})} \int_{\Tilde{K}}\phi^{p-1}  |\nabla u|^{p-1}|\nabla\phi|\,dx.
         \end{split}
     \end{equation*}
     Now let $\epsilon>0$. Applying Young's inequality with the parameter $\epsilon$ for the second integral on the right side
     \begin{equation*}
         \int_{\Tilde{K}}\phi^{p-1}  |\nabla u|^{p-1}|\nabla\phi|\,dx\leq \frac{\epsilon(p-1)}{p}\int_{\Tilde{K}}\phi^{p}|\nabla u|^{p}\,dx +\frac{1}{p\epsilon^{p-1}}\int_{\Tilde{K}}|\nabla\phi|^{p}\,dx.
     \end{equation*}
     Choosing $\epsilon=\frac{\theta_{1}^{p}p}{2(p-1)}$, we obtain 
     \begin{equation*}
         \int_{K}|\nabla u|^{p}\,dx\leq C(p,K,\theta_{1},\theta_{2})\|u\|_{L^{\infty}(\Tilde{K})}\left(f(\|u\|_{L^{\infty}(\Tilde{K})})+1\right).\qedhere
     \end{equation*}
 \end{proof}
 Before proving the existence of a solution on $\mathcal{W_{R}}$, we  will consider radial reduction.

\begin{lemma}\label{lemma-radial-reduction}
    Consider $0\leq R_{1}<R_{2}$. Let $A_{R_{1}}^{R_{2}}:=\mathcal{W}_{R_{2}}(0)\setminus\mathcal{W}_{R_{1}}(0)$ if $R_{1}>0$ and $A_{R_{1}}^{R_{2}}:=\mathcal{W_{R_{2}}}(0)$ if $R_{1}=0$. If $u$ is a radial solution of \eqref{finsler eqn} on $A_{R_{1}}^{R_{2}}$. Then $w(H_{0}(x)):=u(x)$ is the solution of 
    \begin{equation*}
        \left(t^{n-1}|w'(t)|^{p-2}w'(t)\right)'=t^{n-1}f(w)\quad \text{on }(R_{1},R_{2}).
    \end{equation*}
    The converse also holds.
    \end{lemma}
\begin{proof}
   For any $\phi\in C_{c}^{\infty}(A_{R_{1}}^{R_{2}}(0))$ such that $\phi(x)=\psi(H_{0}(x))$,
       \begin{equation*}
           -\int_{A_{R_{1}}^{R_{2}}(0)} \left<|w'(H_{0}(x))|^{p-1}H^{p-1}(\nabla H_{0}(x))\frac{w'(H_{0}(x))}{|w'(H_{0}(x))|}\nabla H(\nabla H_{0}(x)),\psi' (H_{0}(x))\nabla H_{0}(x)\right>=\int_{A_{R_{1}}^{R_{2}}(0)}f(v)\phi \,dx.
       \end{equation*}
       Using Lemma \ref{lemma_properties of H}\ and \ref{H0},
       \begin{equation*}
                      -\int_{A_{R_{1}}^{R_{2}}(0)} |w'(H_{0}(x))|^{p-2} w'(H_{0}(x)) \psi' (H_{0}(x)) \,dx=\int_{A_{R_{1}}^{R_{2}}(0)}f(w(H_{0}(x))\psi(H_{0}(x)) \,dx.
       \end{equation*}
       Changing variables, we obtain 
       \begin{equation*}
           -\int_{R_{1}}^{R_{2}}\int_{\mathcal{W}_{t}(0)}|w'(t)|^{p-2}w'(t)\psi'(t) \frac{\,d\mathcal{H}(x)}{|\nabla H_{0}(x)|}\,dt=\int_{R_{1}}^{R_{2}}\int_{\mathcal{W}_{t}(0)}f(w(t))\psi(t) \frac{\,d\mathcal{H}(x)}{|\nabla H_{0}(x)|}\,dt.
       \end{equation*}
       After the cancellation of the common terms, we get
       \begin{equation*}
          - \int_{R_{1}}^{R_{2}} t^{n-1}|w'(t)|^{p-2}w'(t)\psi'(t) \,dt=\int_{R_{1}}^{R_{2}} t^{n-1}f(w(t))\psi(t) \,dt.
       \end{equation*} 

       To prove the converse, let $w$ be a solution on the one-dimensional equation. Define $u(x)=w\left(H_{0}(x)\right)$. then by Lemma \ref{H0}
       \begin{equation*}
           H^{p-1}(\nabla u)\nabla H(\nabla u)
=
|w'(H_0(x))|^{p-2}w'(H_0(x))
\frac{x}{H_0(x)}.
       \end{equation*}
Let $\phi\in C_c^\infty(A_{R_1}^{R_2})
$. Therefore, by coarea formula we have 
\begin{equation*}
\begin{split}
   -\int_{A_{R_1}^{R_2}}
\left\langle
H^{p-1}(\nabla u)\nabla H(\nabla u),
\nabla\phi
\right\rangle\,dx
&=
-\int_{A_{R_1}^{R_2}}
|w'(H_0(x))|^{p-2}w'(H_0(x))
\left\langle
\frac{x}{H_0(x)},\nabla\phi(x)
\right\rangle\,dx \\
&=-\int_{R_1}^{R_2}
|w'(r)|^{p-2}w'(r)
\int_{\{H_0=r\}}
\frac{
\left\langle \frac{x}{r},\nabla\phi(x)\right\rangle
}
{|\nabla H_0(x)|}
\,d\mathcal{H}(x)\,dr \\
&=-\int_{R_1}^{R_2}
|w'(r)|^{p-2}w'(r)\int_{\mathcal{W}_{1}}\frac{
\left\langle y,\nabla\phi(ry)\right\rangle
}
{|\nabla H_0(y)|}r^{n-1}
\,d\mathcal{H}(y)\,dr.
\end{split} 
\end{equation*}
Define
\begin{equation*}
    \eta(r):=
\int_{\mathcal{W}_{1}}
\frac{\phi(ry)}
{|\nabla H_0(y)|}\,
d\mathcal{H}(y),
\qquad r\in(R_1,R_2).
\end{equation*}
Then $\eta\in C_c^1((R_1,R_2))$ and by differentiating under the integral sign one has
\begin{equation*}
    \eta'(r)
=
\int_{\mathcal{W}_{1}}
\frac{\langle \nabla\phi(ry),y\rangle}
{|\nabla H_0(y)|}\,
d\mathcal{H}(y).
\end{equation*}
Since $w$ is a solution
\begin{equation*}
    -\int_{R_1}^{R_2}
r^{n-1}|w'(r)|^{p-2}w'(r)\eta'(r)\,dr
=
\int_{R_1}^{R_2}
r^{n-1}f(w(r))\eta(r)\,dr.
\end{equation*}
Similarly, one can compute
\begin{equation*}
    \begin{split}
       \int_{R_1}^{R_2}
r^{n-1}f(w(r))\eta(r)\,dr
&=
\int_{R_1}^{R_2}
r^{n-1}f(w(r))
\int_{\mathcal{W}_{1}}
\frac{\phi(ry)}
{|\nabla H_0(y)|}
\,d\mathcal{H}(y)\,dr
\\
&=
\int_{R_1}^{R_2}
f(w(r))
\int_{\{H_0=r\}}
\frac{\phi(x)}
{|\nabla H_0(x)|}
\,d\mathcal{H}(x)\,dr
\\
&=
\int_{A_{R_1}^{R_2}}
f(w(H_0(x)))\phi(x)\,dx
\\
&=
\int_{A_{R_1}^{R_2}}
f(u(x))\phi(x)\,dx. 
    \end{split}
\end{equation*}

Thus,
\begin{equation*}
  -\int_{A_{R_1}^{R_2}}
\left\langle
H^{p-1}(\nabla v)\nabla H(\nabla v),
\nabla\phi
\right\rangle\,dx
=
\int_{A_{R_1}^{R_2}}
f(v)\phi\,dx  
\end{equation*}
This proves the converse.

\end{proof}
Since $\Omega$ is an open set, given $x\in \Omega$ choose $R>0$ such that $B_{R}(x)\subset \Omega$. Since $|x-y|\leq \theta_{2}H_{0}(x-y)$, take any $r\leq R/\theta_{2}$, then $\mathcal{W}_{r}(x)\subset B_{R}(x)\subset\Omega$. Using this, we construct a local barrier to a large solution, which is radial on a Wulff ball. For this reason we first establish Theorem \ref{main theorem existence} on a Wulff ball $\mathcal{W}_{R}$.
  \begin{theorem}\label{thrm existence on a wolff ball}
      Assume \ref{KO}, \ref{a1}. Then \eqref{finsler large eqn} admits a solution when $\Omega=\mathcal{W}_{R}(0)$. 
  \end{theorem}
  \begin{proof}
      Let $k\in \mathbb{N}$, consider 
      \begin{equation*}
          \begin{cases}
              \text{div}\left(H^{p-1}(\nabla v_{k})\nabla H(\nabla v_{k})\right)=f(v_{k})\quad &\text{in }\mathcal{W}_{R}(0)\\
        v_{k}(x)=k &\text{as dist}(x,\partial\mathcal{W}_{R}(0))\to 0.
          \end{cases}
      \end{equation*}
      If $v_{k}(x)=w_{k}\left(H_{0}(x)\right)$ is radial, then by Lemma \ref{lemma-radial-reduction} $w_{k}$ solves
      \begin{equation}\label{eqn for approximating seq of w on wulff ball}
          \begin{cases}
              \left(t^{n-1}|w_{k}'(t)|^{p-2}w_{k}'(t)\right)'=t^{n-1}f(w_{k}) \quad &\text{on }(0,R)\\
              w_{k}(R)=k, ~ w_{k}'(0)=0.
          \end{cases}
      \end{equation}
  The existence of $w_{k}$ follows by \cite[Theorem 6.2]{PucciSerrin2006}, in addition, $w_{k}'\geq 0$. Integrating the equation between $0$ and $t$, one can get
  \begin{equation*}
      |w_{k}'(t)|^{p-2}w_{k}'(t)=\frac{1}{t^{n-1}}\int_{0}^{t}s^{n-1}f(w_{k}(s))\,ds\leq \frac{t}{n}f(w_{k}(t)).
  \end{equation*}
  
  Further, rewriting the equation that $w_{k}$ solves, we get
  \begin{equation*}
      (|w_{k}'(t)|^{p-2}w_{k}'(t))'+\frac{n-1}{t}|w_{k}'(t)|^{p-2}w_{k}'(t)=f(w_{k}(t)) \quad \text{for all }t\in (0,R).
  \end{equation*}
  Combining the previous two equations one can get
  \begin{equation*}
      (|w_{k}'(t)|^{p-2}w_{k}'(t))'\geq \frac{1}{n} f(w_{k}(t)) \quad \text{for all }t\in (0,R).
  \end{equation*}
  As in the previous section, multiplying the above equation with $w_{k}'$ and integrating from $0$ to $t$ we get
  \begin{equation*}
      (p-1)\int_{0}^{w_{k}'(t)}s^{p-1}\,ds\geq \frac{1}{n}(F(w_{k}(t))-F(w_{k}(0)).
  \end{equation*}
  This implies $n(p-1) (w_{k}'(t))^{p}\geq p (F(w_{k}(t))-F(w_{k}(0)))$. Taking the power $1/p$, then rearranging the equation and integrating from $s\geq 0$ to $r\leq R$, we get
  \begin{equation}\label{eqn to bound finite solutions on ball}
      \left(\frac{p-1}{p}\right)^{1/p}\int_{w_{k}(s)}^{w_{k}(r)}\frac{dt}{\{F(t)-F(w_{k}(0))\}^{1/p}}\geq \frac{(r-s)}{n^{1/p}}.
  \end{equation}
 Define 
 \begin{equation*}
     \Psi(a):=\left(\frac{p-1}{p}\right)^{1/p}\int_{a}^{\infty}\frac{dt}{\{F(t)-F(a)\}^{1/p}}= \left(\frac{p-1}{p}\right)^{1/p}\int_{0}^{\infty}\frac{dt}{\{F(t+a)-F(a)\}^{1/p}}
 \end{equation*}
   
As in the second step of Theorem \ref{thrm existence in one dim with osgood}, $\Psi(a)\to L\leq +\infty$ (depending on \ref{i1} or \ref{i2}) and $\Psi(a)\to 0$ as $a\to +\infty$.

  Thus taking $s=0$ and $r/n^{1/p}<L$, we get that 
  \begin{equation*}
      w_{k}(0)\leq M:= \Psi^{-1}\left(\frac{r}{n^{1/p}}\right).
  \end{equation*}
Choose $t$ large such that $F(t)>2F(M)$, we get
\begin{equation*}
    \frac{1}{\{F(t)-F(w_{k}(0))\}^{1/p}}\leq \frac{2^{1/p}}{F(t)^{1/p}}.
\end{equation*}
Further, for a fixed $k$, if $w_{k}$ is not bounded, then there exists $s_{n}$ such that $w_{k}(s_{n})\to +\infty$ as $n\to +\infty$. Thus for $n$ large, using \eqref{eqn to bound finite solutions on ball} with $r=R$, we get
\begin{equation*}
    \Psi_{H,p}(w_{k}(s_{n}))\geq \left(\frac{p-1}{p}\right)^{1/p}\int_{w_{k}(s_{n})}^{k}\frac{dt}{\{F(t)\}^{1/p}}\geq \frac{(R-s_{n})}{n^{1/p}}.
\end{equation*}
By \ref{KO}, $\Psi_{H,p}(w_{k}(s_{n}))\to 0$ as $n\to +\infty$. So $s_{n}\to R$, a contradiction as $w_{k}(R)=k$. 

Thus, for $s$ such that $R-s<L$, we can write using \eqref{eqn to bound finite solutions on ball}, fact that $\Psi$ is decreasing, and the monotonicity of $w_{k}$, that for any $t\leq s$
\begin{equation*}
   w_{k}(t)\leq w_{k}(s)\leq \Psi^{-1}\left(\frac{R-s}{n^{1/p}}\right).
\end{equation*}
   Here, the bound is independent of 
   $k$. Thus the sequence $\{v_{k}(x):=w_{k}\left(H_{0}(x)\right)\}$ is locally uniformly bounded.
  By the first part of the Theorem \ref{comparison principle}, the sequence $\{v_{k}(x)=w_{k}\left(H_{0}(x)\right)\}_{k}$ is increasing. Define $w(s):=\sup_{k}w_{k}(s)$ and $v(x)=w\left(H_{0}(x)\right)$ on $\mathcal{W}_{R}(0)$. Since $v(x)\geq v_{k}(x)$ for all $x\in \mathcal{W}_{R}(0)$, we get $\liminf_{H_{0}(x)\to R}v(x)\geq k$ for all $k$. Thus, $v(x)\to +\infty$ as $H_{0}(x)\to R$.

It remains to show that the limit $v$ is a weak solution. For this, we repeat the compactness and Browder-Minty type argument of Steps 3-4 in the proof of Theorem \ref{main theorem existence} (given below). The only a priori ingredients needed there are a local $L^\infty$ bound and a local $W^{1,p}$ bound; the former has been established above independently, while the latter follows from Proposition~\ref{prop_grad lp bound}. Hence, $v$ is a weak solution of \eqref{finsler large eqn} in $\mathcal{W}_R(0)$.
  \end{proof}

  This existence will be used to construct a local barrier which will be used in the proof of existence of large solutions on $\Omega$ other than the ball $\mathcal{W}_{R}$, $R>0$.
  \begin{proposition}\label{prop_weak_local_bound}
    Let $\Omega$ be a bounded domain in $\rn$, $n\geq 2$. Assume \ref{KO}, \ref{a1}. Let $u\in W_{loc}^{1,p}(\Omega)\cap C(\Omega)$ be a weak subsolution of \eqref{finsler eqn}. For any $x_{0}\in \Omega$ and $R>0$ such that $\mathcal{W}_{R}(x_{0})\subset \subset \Omega$, there exists a positive constant $C(R,n,p,f)$, not depending on $u$ such that $u\leq C(R,n,p)$ on $\mathcal{W}_{\frac{R}{2}}(x_{0})$.
\end{proposition}
 \begin{proof}
 Let $w$ be the solution of \eqref{finsler large eqn} for $\Omega=\mathcal{W}_{R}(x_{0})$, since $u$ is finite on $\partial\mathcal{W}_{R}(x_{0})$, by the first part of the Theorem \ref{comparison principle}, $u\leq w$ on $\mathcal{W}_{R}(x_{0})$. The result follows.
 \end{proof}
 \begin{remark}
     The argument in the above proposition is independent of the boundary values. Thus, it holds for solutions with finite Dirichlet boundary values as well as infinite boundary values. 
 \end{remark}

 Next, we address the existence:
 \begin{proof}[\textbf{Proof of Theorem \ref{main theorem existence} for $n\geq 2$}]
 If $\Omega$ is an Wulff ball, then the existence is proved independently in Theorem \ref{thrm existence on a wolff ball}. We assume that $\Omega$ is
   not a Wullf ball.
   
   The idea is to get the large solution as a limit of a sequence of solutions with finite boundary data.
   
  \textit{Step 1:} Let $u_{k}$ be the solution to \eqref{dirchlet problem} as in Definition \ref{dfn weak solution to dirchlet prob} with $g=k$. Then $u_{k}$ is also a solution in the sense of Definition \ref{dfn solution}. By Theorem \ref{comparison principle}, the sequence $\{u_{k}\}_{k}$ is increasing and by Proposition \ref{prop_weak_local_bound} it is locally uniformly bounded. For any $x\in\Omega$, define
   \begin{equation}\label{dfn of u}
       u(x):=\lim_{k\to\infty}u_{k}(x).
   \end{equation}
   By Proposition \ref{prop_weak_local_bound}, $u\in L_{loc}^{\infty}(\Omega)$. Then, by the dominated convergence theorem $u_{k}\to u$ in $L^{q}_{loc}(\Omega)$ for all $1\leq q<\infty$. Take $\Omega'\subset\subset\Omega$, by Proposition \ref{prop_weak_local_bound}
 and Proposition \ref{prop_grad lp bound}, the sequence $\{u_{k}\}_{k}$ is bounded in $W^{1,p}(\Omega')$ and hence $u_{n_{k}}\rightharpoonup v$ in $W_{loc}^{1,p}(\Omega')$ for some $v\in W^{1,p}(\Omega')$. By the continuous embedding of $W^{1,p}(\Omega')$ in $L^{p}(\Omega')$ and the strong $L^{p}(\Omega')$ convergence of $u_{k}$ gives $v=u$. Thus, $u\in W_{loc}^{1,p}(\Omega)$. 
 
 We show in three steps that the candidate $u$ defined above is the required solution.
 
\textit{Step 2:}
 The continuity of $f$ implies $f(u_{k}(x))\to f(u(x))$ for all $x\in\Omega$. The monotonicity of $f$ and the monotone convergence theorem imply $f(u_{k})\to f(u)$ in $L^{1}_{loc}(\Omega)$, thus $f(u)\in L^{1}_{loc}(\Omega)\cap L^{\infty}_{loc}(\Omega)$.  

 Let $\Omega'\subset \subset \Omega$. For any $\phi\in W_{0}^{1,p}(\Omega')$, by the dominated convergence theorem 
 \begin{equation*}
     \int_{\Omega'}f(u_{k})\phi\,dx\to \int_{\Omega'}f(u)\phi\,dx.
 \end{equation*}

\textit{Step 3:}
     For any $v\in W^{1,p}(\Omega')$, by Lemma \ref{lemma_properties of H} and the H\"older inequality, we obtain 
    \begin{equation*}
    \begin{split}
         \int_{\Omega'}\left<\Fin{u_{n_{k}}}, \nabla v\right> \,dx & \leq C \int_{\Omega'}H^{p-1}(\nabla u_{n_{k}}) |\nabla v| \,dx \\
         & \leq C\theta_{2}^{p-1}\int_{\Omega'} |\nabla u_{n_{k}}|^{p-1}|\nabla v|\,dx \\
         &\leq  C\theta_{2}^{p-1} \|\nabla u_{n_{k}}\|^{p-1}_{L^{p}(\Omega')} \|\nabla v\|_{L^{p}(\Omega')}.
    \end{split}
    \end{equation*}
    Therefore, by Propositions \ref{prop_grad lp bound} and \ref{prop_weak_local_bound}, the sequence $\mathcal{T}_{k}\in \left(W^{1,p}(\Omega')\right)'$ given by 
    \begin{equation*}
        \left<\mathcal{T}_{k}, v\right>:=\int_{\Omega'}\left<\Fin{u_{n_{k}}} , \nabla v\right>\,dx
    \end{equation*}
    is bounded and hence there exists a subsequence denoted by $u_{n_{k}}$ and a $\mathcal{T}\in \left(W^{1,p}(\Omega')\right)'$ such that 
    \begin{equation}\label{eqn weak convergence in dual}
        \int_{\Omega'} \left<\Fin{u_{n_{k}}}, \nabla v\right> \,dx=\left<\mathcal{T}_{n_{k}}, v\right> \to \left<\mathcal{T}, v\right> \quad \text{for all } v\in  W^{1,p}(\Omega').
    \end{equation}
    Moreover, for $\phi\in W_{0}^{1,p}(\Omega')$
    \begin{equation*}
        \int_{\Omega'} \left<\Fin{u_{n_{k}}}, \nabla \phi \right>\,dx= -\int_{\Omega'}f(u_{n_{k}})\phi \,dx \to -\int_{\Omega'}f(u)\phi \,dx,
    \end{equation*}
    as $n_{k}\to \infty$. Thus, 
    \begin{equation}\label{L equals to f}
        \left<\mathcal{T}, \phi\right>=-\int_{\Omega'}f(u)\phi \,dx \quad \text{for all } \phi\in W_{0}^{1,p}(\Omega').
    \end{equation}

    \textit{Step 4:}
    Finally, we will show that
    \begin{equation*}
        \left<\mathcal{T}, \phi\right>=\int_{\Omega'} \left<\Fin{u}, \nabla \phi \right> \,dx \quad \text{for all } \phi \in W_{0}^{1,p}(\Omega'),
    \end{equation*}
     using Browder-Minty type argument (cf. \cite{Zeidler1990,evans}).
     
    Let $\Omega'\subset\subset\tilde{\Omega}\subset\subset\Omega$ and choose $\eta\in C_{c}^{\infty}(\tilde{\Omega})$ such that $\eta=1$ on $\Omega'$ and $0\leq \eta\leq 1$.
     Then Lemma \ref{lemma_properties of H} and the strict monotonicity imply
    \begin{equation*}
    \begin{split}
         \int_{\Omega'}|\nabla(u_{n_{k}}-u)|^{p}\,dx&\leq C \int_{\tilde{\Omega}}\eta^{p}H^{p}(\nabla(u_{n_{k}}-u))\,dx\\
         &\leq C\int_{\tilde{\Omega}}\eta^{p}\left<\Fin{u_{n_{k}}}-\Fin{u}, \nabla(u_{n_{k}}-u)\right> \,dx.
    \end{split}
       \end{equation*}
       By $u_{n_{k}}\rightharpoonup u$, the second integral 
       \begin{equation*}
           \int_{\tilde{\Omega}}\eta^{p}\left<\Fin{u}, \nabla(u_{n_{k}}-u)\right> \,dx\to 0\quad \text{as }n_{k}\to \infty.
       \end{equation*}
       Further, take $\psi_{n}:=\eta^{p}(u_{n_{k}}-u)\in W_{0}^{1,p}(\tilde{\Omega})$. Then the first integral is
       \begin{equation*}
       \begin{split}
    \int_{\tilde{\Omega}}\eta^{p}\left<\Fin{u_{n_{k}}}, \nabla(u_{n_{k}}-u)\right> &= \int_{\tilde{\Omega}}\eta^{p}\left<\Fin{u_{n_{k}}},\{\nabla\psi_{n_{k}}- (u_{n_{k}}-u)\nabla\eta^{p}\}\right> \,dx\\
           &=-\int_{\tilde{\Omega}}f(u_{n_{k}})\eta^{p}(u_{n_{k}}-u)\,dx-\\
           &\hspace{1.5cm} -\int_{\tilde{\Omega}}(u_{n_{k}}-u)\left<\Fin{u_{n_{k}}}, \nabla \eta^{p}\right>\,dx.
       \end{split}
   \end{equation*}
   Since $u_{n_{k}}\to u$ in $L^{p}(\tilde{\Omega})$, the absolute value of the first integral converges to $0$ in view of Proposition \ref{prop_weak_local_bound} with the monotonicity of $f$, and the absolute value of the second integral converges  to $0$ in view of the bound of $\{u_{n_{k}}\}$ in $W^{1,p}(\tilde{\Omega})$. Thus $u_{n_{k}}\to u$ in $W^{1,p}_{loc}(\Omega)$.

    By (monotonicity) Lemma \ref{lemma monotone operator}
    \begin{equation*}
        \int_{\Omega'} \left<\Fin{u_{n_{k}}}-\Fin{v}, \nabla(u_{n_{k}}-v) \right> \,dx \geq 0, ~~ \text{for any } v \in W^{1,p}(\Omega').
   \end{equation*}
   By $u_{n_{k}}\to u$ in $W^{1,p}(\Omega')$ and the bound of $\{u_{n_{k}}\}$ in $W^{1,p}(\tilde{\Omega})$ we obtain 
   \begin{equation*}
       \int_{\Omega'}\left<\Fin{u_{n_{k}}}, \nabla(u_{n_{k}}-u)\right>\,dx\to 0.
   \end{equation*}
   By \eqref{eqn weak convergence in dual} we obtain 
   \begin{equation*}
       \int_{\Omega'}\left<\Fin{u_{n_{k}}}, \nabla(u-v)\right>\,dx\to \left<\mathcal{T},u-v\right>.
   \end{equation*}
   And $u_{n_{k}}\rightharpoonup u$ we obtain 
   \begin{equation*}
       \int_{\Omega'}\left<\Fin{v},\nabla (u_{n_{k}}-v)\right>\,dx\to \int_{\Omega'}\left<\Fin{v},\nabla (u-v)\right>\,dx.
   \end{equation*}
  Combining all three convergences above, we obtain with an abuse of notation    
   \begin{equation*}
       \left<\mathcal{T}-\int_{\Omega'}\Fin{v}, u-v\right>\geq 0.
   \end{equation*}
   Now, for any $w\in W^{1,p}(\Omega')$ substitute $v=u-tw$, $t\in [0,1]$ to get 
   \begin{equation}\label{BM}
       \left<\mathcal{T}-\int_{\Omega'}\Fin{u-tw}, w\right>\geq 0.
   \end{equation}
   Also, the map $t\longmapsto \int_{\Omega'}\left<\Fin{u-tw}, \nabla w\right>$ is continuous by using $H\in C^{2}(\rn\setminus\{0\})$ when $\nabla u\neq 0$ and the 1-homogeneity of $H$ otherwise. For any $w\in W^{1,p}(\Omega')$, $t\to 0$ in \eqref{BM} implies 
   \begin{equation*}
       \left<\mathcal{T}-\int_{\Omega'}\Fin{u}, w\right>\geq 0.
   \end{equation*}
   Replacing $w$ with $-w$, we infer that 
   \begin{equation*}
\left<\mathcal{T},w\right>=\int_{\Omega'}\left<\Fin{u}, \nabla w\right>\,dx.
   \end{equation*}
    This together with \eqref{L equals to f} implies 
   \begin{equation}
       -\int_{\Omega'}\left<\Fin{u}, \nabla \phi \right>\,dx = \int_{\Omega'}f(u)\phi \,dx \quad \text{for all } \phi\in W_{0}^{1,p}(\Omega').
   \end{equation}
   Thus $u$ solves the equation. 

   Finally, since $u\geq u_{k}$ on $\Omega$ for all $k$ and since $u_{k}\in C(\overline{\Omega})$ (see Proposition \ref{prop-regularity}) the boundary condition is satisfied by taking limit inferior.
\end{proof}

\section{Boundary Asymptotics}\label{asym}
In this section, we prove Theorem \ref{dual asym} in two steps. These are Lemma \ref{lemma asym 1} and Lemma \ref{lemma asym 2}. In order to prove Lemma \ref{lemma asym 2}, we are led to prove the existence as well as the boundary behaviour of a radially symmetric (with respect to the dual norm ${H_{0}}$) solution on an annulus defined with respect to $H_{0}$. Since Theorem \ref{dual asym} is proved for $n=1$ in Section \ref{1d}, we assume $n\geq 2$ throughout this section.

  Let $\Omega\subset\R^{n}$, be a bounded domain with $C^{2}$ boundary. Let $\delta_{H_{0}}(x)=\inf_{z\in \partial\Omega}H_{0}(x-z)$ denote the distance from the boundary of $\Omega$. We first collect some of the properties of $\delta_{H_{0}}$ found in \cite{CrastaMalusa2007}. 

Since $\Omega$ is $C^{2}$, there exists $\mu>0$ such that $\delta_{H_{0}}$ is $C^{2}$ on $\Omega_{\mu}:=\{x\in\Omega~|~\delta_{H_{0}}(x)<\mu\}$ and $\Omega_{\mu}':=\{y\in \rn\setminus \overline{\Omega}~|~ \delta_{H_{0}}(y)<\mu\}$, see \cite[Theorem 4.16]{CrastaMalusa2007}. Moreover, for any $x\in \Omega_{\mu}\cup \Omega_{\mu}'$, there is a unique $z(x)\in \partial\Omega$ such that
 \begin{equation}\label{rep of x in Omega mu}
 \begin{split}
    &\delta_{H_{0}}(x)=H_{0}(x-z(x)) \quad\text{and}\\
   & x=z(x)+\delta_{H_{0}}(x)\nabla H(\nabla\delta_{H_{0}}(x)).
 \end{split}
       \end{equation} 
Further, there exists $R>0$, such that given any $z\in \partial\Omega$, there exists $x^{int}(z)\in \Omega$ and $x^{ext}(z)\in \R^{n}\setminus\Omega$, such that $\mathcal{W}_{R}(x^{int}(z))\subset\Omega$ with $\partial\mathcal{W}_{R}(x^{int}(z))\cap \partial\Omega=\{z\}$ and $\mathcal{W}_{R}\left(x^{ext}(z)\right)\subset \rn\setminus \overline{\Omega}$ with $\partial\mathcal{W}_{R}\left(x^{ext}(z)\right)\cap \partial\Omega=\{z\}$ (By \cite[Remark 4.2, Proposition 3.3 and also Proposition 4.6]{CrastaMalusa2007}. Moreover,
\begin{equation}\label{rep of center of ball}
    x^{int}=z+R\nabla H\left(\nabla \delta_{H_{0}}(x^{int})\right)\quad \text{and}\quad x^{ext}= z+R\nabla H\left(\nabla \delta_{H_{0}}(x^{ext})\right).
\end{equation}
Indeed, one has to choose $x^{int}\in \Omega_{\mu}$ and $x^{ext}\in \Omega_{\mu}'$. Finally, \cite[Lemma 4.1]{CrastaMalusa2007} implies that bounded domain $\Omega$ with $C^{2}$ boundary has uniform interior and exterior ball conditions with respect to $H_{0}$ norm.

\begin{proposition}\label{prop_boundary asym for radial solution}
    Let $v$ be the solution \eqref{finsler large eqn} on $\mathcal{W}_{R}(0)$ given in Theorem \ref{thrm existence on a wolff ball}, then 
    \begin{equation*}
          \liminf_{H_{0}(x)\to R}\frac{\Psi_{H,p}\left(v(x)\right)}{R-H_{0}(x)}\geq 1.
    \end{equation*}
\end{proposition}
\begin{proof}
The idea of the proof is as follows: We consider $\Psi_{H,p}$ with $r=w(H_{0}(x)):=v(x)$ and apply a change of variable to obtain $\int_{t}^{\infty}w'(s)/\{F(w(s)\}^{1/p}~ds$. Using the equation and the symmetry, one can compare the quantities $w'$ and $F(w)$ present in the integral, to obtain the desired inequality.

    By the definition of $v(x)=w(H_{0}(x))$, where $w$ solves
    \begin{equation*}
        \begin{cases}
              \left(t^{n-1}|w'(t)|^{p-2}w'(t)\right)'=t^{n-1}f(w) \quad &\text{on }(0,R)\\
              w(t)\to+\infty \text{ as }t\to R, ~ w'(0)=0.
          \end{cases}
    \end{equation*}
    Since $w(t)=\sup_{k}w_{k}(t)$, where $w_{k}$ solve \eqref{eqn for approximating seq of w on wulff ball} and $w_{k}'\geq 0$, we get for $s<t$ that $w_{k}(s)\leq w_{k}(t)$. Passing the limit $k\to +\infty$ we obtain $w(s)\leq w(t)$ for $s<t$. Thus $w'\geq 0$.
    Equivalently,
    \begin{equation*}
    \left(|w'(t)|^{p-2}w'(t)\right)'+\frac{(n-1)}{t}|w'(t)|^{p-2}w'(t)=f(w(t)).
\end{equation*}
Here we follow the proof presented in \cite[Proposition 4.3]{materopaper}. The above equation can be written as
\begin{equation*}
    \left((w'(t))^{p-1}\right)'+\frac{(n-1)}{t}(w'(t))^{p-1}=f(w(t)).
\end{equation*}
Which implies
\begin{equation*}
    (p-1)(w'(t))^{p-2}w''(t)+\frac{(n-1)}{t}(w'(t))^{p-1}=f(w(t)).
\end{equation*}
    Multiplying by $\frac{p}{p-1}(w'(t))t^{\frac{p(n-1)}{p-1}}$, we obtain 
    \begin{equation*}
        \left(t^{\frac{p(n-1)}{p-1}}w'(t)^{p}\right)'=\frac{p}{p-1}t^{\frac{p(n-1)}{p-1}}w'(t)f(w(t)).
    \end{equation*}
    Integrating from $t_{0}> 0$ to $t<R$
\begin{equation*}
     w'(t)^p
=  \left(\frac{t_0}{t}\right)^{\frac{p(n-1)}{p-1}} w'(t_0)^p +\frac{p}{p-1}  \int_{t_0}^{t} \left(\frac{s}{t}\right)^{\frac{p(n-1)}{p-1}}  w'(s) f(w(s)) \, ds \, .
\end{equation*}
Adding and subtracting $\frac{p}{p-1}\int^{t}_{t_{0}}dF(w(s))$, we can rewrite the above equation as
    \begin{equation*}
    \begin{split}
       w'(t)^p
&=  \left(\frac{t_0}{t}\right)^{\frac{p(n-1)}{p-1}} w'(t_0)^p + \frac{p}{p-1} \left(F(w(t)) - F(w(t_0))\right) +\frac{p}{p-1}  \int_{t_0}^{t} \left(\left(\frac{s}{t}\right)^{\frac{p(n-1)}{p-1}}-1\right)  \, dF(w(s)) \\ 
&\geq \frac{p}{p-1} \left(F(w(t)) - F(w(t_0))\right) +\frac{p}{p-1} \left(\left(\frac{t_{0}}{R}\right)^{\frac{p(n-1)}{p-1}}-1\right)\left(F(w(t)) - F(w(t_0))\right)\\
&\geq  \frac{p}{p-1} \left(\frac{t_{0}}{R}\right)^{\frac{p(n-1)}{p-1}}\left(F(w(t)) - F(w(t_0))\right).
    \end{split}
\end{equation*}
Given $1>\epsilon>0$, choose $0<t_{0}<R$ such that $(t_{0}/R)^{p(n-1)/(p-1)}>(1-\epsilon)$. Then for all $t_{0}<t<R$, we have
\begin{equation*}
     w'(t)^p\geq \frac{p}{p-1}(1-\epsilon) \left(F(w(t)) - F(w(t_0))\right).
\end{equation*}

Because $w(s)\to+\infty$ as $s\to R$, choose $t_{1}$, $t_{0}\leq t_{1}<R$, such that $F(w(s))>(1/\epsilon)F(w(t_{0}))$ for all $s>t_{1}$. Thus for all $t_{1}<t<R$, we get
\begin{equation*}
    w'(t)\geq (1-\epsilon)^{2/p} \left\{\frac{p}{p-1}F(w(t))\right\}^{1/p}.
\end{equation*}
Thus, for $t_{1}<t<R$, since $w(t)\to +\infty$ as $t\to R$, we get
\begin{equation*}
    \begin{split}
        \Psi_{H,p}(w(t))&=\left(\frac{p-1}{p}\right)^{1/p}\int_{t}^{R}\frac{w'(s)}{F(w(s))^{1/p}}\,ds\\
        &\geq (1-\epsilon)^{2/p}(R-t)
    \end{split}
\end{equation*}
Since $\epsilon>0$ is arbitrary, we infer
\begin{equation*}
   \liminf_{H_{0}(x)\to R} \frac{\Psi_{H,p}(v(x))}{R-H_{0}(x)}=\liminf_{H_{0}(x)\to R} \frac{\Psi_{H,p}\left(w(H_{0}(x))\right)}{R-H_{0}(x)}\geq 1.
\end{equation*}
\end{proof}
   \begin{lemma}\label{lemma asym 1}
       Let $\Omega\subset\rn$ be a bounded domain with $C^{2}$ boundary. Let $u$ be a solution of \eqref{finsler large eqn}.  Then 
       \begin{equation*}
           \liminf_{\delta_{H_{0}}(x)\to 0}\frac{\Psi_{H,p}\left(u(x)\right)}{\delta_{H_{0}}(x)}\geq 1.
       \end{equation*}
   \end{lemma}
   \begin{proof}
   Let $\mu>0$ be such that for all $x\in \Omega_{\mu}$ there exists $z(x)\in \partial\Omega$ such that \eqref{rep of x in Omega mu} holds. Denoting $z=z(x)$, let $\mathcal{W}_{R}(x^{int})$ be an interior $H_{0}$-ball associated with $z$ and let $v_{r}$ be the solution of \eqref{finsler large eqn} on $\mathcal{W}_{R-r}(x^{int})$ for $0<r<R$.
   
   As in the proof of Proposition \ref{prop_weak_local_bound}, $u\leq v_{r}$ on $\mathcal{W}_{R-r}(x^{int})$. By Theorem \ref{comparison principle}, $\{v_{r}(x)\}_{r\leq r_{0}}$ is decreasing as $r\to 0$, for a fixed $x\in \mathcal{W}_{R}(x^{int})$ and $r_{0}$ sufficiently small. Moreover, this sequence is pointwise bounded below by the solution of \eqref{finsler large eqn} on $\mathcal{W}_{R}(x^{int})$. As $v_{r_{0}}$ is locally bounded, by the dominated convergence theorem, $v_{r}\to v:=\inf_{r}v_{r}$ in $L^{p}_{loc}(\mathcal{W}_{R}(x^{int}))$ for $0<r<r_{0}$. By Proposition \ref{prop_grad lp bound}, there exists a subsequence denoted by $v_{r}$ such that $v_{r}\rightharpoonup v$ in $W^{1,p}_{loc}(\mathcal{W}_{R}(x^{int}))$. As in the proof of Theorem \ref{main theorem existence}, one obtains that $v(x)$ is the solution of \eqref{finsler large eqn} on $\mathcal{W}_{R}(x^{int})$. We also refer to \cite[Theorem 4.3]{DattaIndro2026} for the use of the monotonicity method in the case when the sequence of solutions, defined in nested domains, blows up at the boundary. Morever, since each $v_{r}$ is radial, so is $v$.
    
    Further, since $u(x)\leq v(x)$ by the above discussion, and $\Psi_{H,p}$ is decreasing, one has $\Psi_{H,p}(u(x))\geq \Psi_{H,p}(v(x))$ for all $x\in \mathcal{W}_{R}(x^{int})$. By \eqref{rep of x in Omega mu} and \eqref{rep of center of ball} for any $x\in \Omega_{\mu}$, we obtain
       \begin{equation*}
           \frac{\Psi_{H,p}\left(u(x)\right)}{\delta_{H_{0}}(x)}=\frac{\Psi_{H,p}\left(u(x)\right)}{H_{0}(x-z(x))}\geq \frac{\Psi_{H,p}\left(v(x)\right)}{H_{0}(x-z(x))}
           = \frac{\Psi_{H,p}\left(v(x)\right)}{R-H_{0}(x-x^{int})}.
            \end{equation*}
      Given $0<\epsilon<1$, by Proposition \ref{prop_boundary asym for radial solution}, there exists $\delta>0$ such that for all $x\in \mathcal{W}_{R}(x^{int})$ with $R-H_{0}(x-x^{int})<\delta$,
      \begin{equation*}
        \frac{\Psi_{H,p}(v(x))}{R-H_{0}(x-x^{int})}\geq 1-\epsilon . 
      \end{equation*}
 Thanks to the uniform ball condition. Choosing $\delta\leq \mu$, for all $x\in \Omega_{\delta}:=\{x\in \Omega~|~\delta_{H_{0}}(x)<\delta\}$,  we infer 
\begin{equation*}
    \frac{\Psi_{H,p}\left(u(x)\right)}{\delta_{H_{0}}(x)}\geq 1-\epsilon.
\end{equation*}
      This implies that 
      \begin{equation*}
          \liminf_{\delta_{H_{0}}(x)\to 0} \frac{\Psi_{H,p}\left(u(x)\right)}{\delta_{H_{0}}(x)}\geq 1.\qedhere
      \end{equation*}
   \end{proof}

    To prove Theorem \ref{dual asym}, it remains to establish the reverse inequality involving the limit superior. To this end, take $0<R_{1}<R_{2}$ and consider the annulus $A_{R_{1}}^{R_{2}}(x_{0})$ defined with respect to the dual norm $H_{0}$, centered at $x_{0}$, with inner radius $R_{1}$ and outer radius $R_{2}$. Take 
   \begin{equation}\label{eqn annulus}
       \begin{cases}
           \Delta_{H}^{p}u=f(u)\quad &\text{in }A_{R_{1}}^{R_{2}}(x_{0})\\
           u(x)\to \infty &\text{as } H_{0}(x-x_{0})\to R_{1}\\
           u(x)= 0 &\text{as } H_{0}(x-x_{0})= R_{2}.
       \end{cases}
   \end{equation}
   The idea of the proofs of the following results are similar to \cite{materopaper} with necessary technical differences.
   The next theorem deals with the existence of a radial solution to the above problem. 
   \begin{theorem}\label{existence on annulus}
       Assume \ref{a1} and \ref{KO}. The equation \eqref{eqn annulus} admits a non-negative radially symmetric solution in $W_{loc}^{1,p}(A_{R_{1}}^{R_{2}}(x_{0}))$.
   \end{theorem}
   \begin{proof}
       Without loss of generality, we assume that the annulus is centred at the origin. We seek a radial solution $v(x)=w(H_{0}(x))$; for $w:[0,\infty)\to [0,\infty)$. Then 
       by Lemma \ref{lemma-radial-reduction}, $w$ solves 
       \begin{equation}\label{eqn radial solution}
           \begin{cases}
               \left(t^{n-1}|w'(t)|^{p-2}w'(t)\right)'=t^{n-1}f(w(t)) \quad\text{ in } (R_{1},R_{2})\\
               w(t)\to \infty \text{ as }t\to R_{1} \text{ and }w(t)\to 0\text{ as }t\to R_{2}.
           \end{cases}
       \end{equation}
       To construct a solution to this, let $k\in \mathbb{N}$ and consider the problem
       \begin{equation*}
           \begin{cases}
               \left((R_{2}-t)^{n-1}|\tilde{w}'(t)|^{p-2}\tilde{w}'(t)\right)'=(R_{2}-t)^{n-1}f(\tilde{w}(t)) \quad\text{ in } (0,R_{2}-R_{1})\\
               \tilde{w}(0)=0 \text{ and }\tilde{w}(R_{2}-R_{1})=k,
           \end{cases}
       \end{equation*}
       which has a $C^{1}(0,R_{2}-R_{1})$ solution $\tilde{w}_{k}\geq 0$ by \cite[Proposition 4.2.1]{PucciSerrinbook2007}. Then $v_{k}(x)=w_{k}(H_{0}(x))=\tilde{w}_{k}(R_{2}-H_{0}(x))$ for $x\in A_{R_{1}}^{R_{2}}(0)$ is the solution of
       \begin{equation*}
           \begin{cases}
           \Delta_{H}^{p}v_{k}=f(v_{k})\quad &\text{in }A_{R_{1}}^{R_{2}}(0)\\
           v_{k}(x)= k &\text{as } H_{0}(x")\to R_{1}\\
           v_{k}(x)= 0 &\text{as } H_{0}(x)\to R_{2}.
       \end{cases}
       \end{equation*}
       The sequence $\{v_{k}\}_{k}$ is increasing by Theorem \ref{comparison principle} and is locally bounded by Proposition \ref{prop_weak_local_bound}. Define $v(x):=\sup_{k}v_{k}(x)$. Then, since Proposition \ref{prop_grad lp bound} also holds for $v_{k}$, as in the proof of Theorem \ref{main theorem existence}, $v$ is the required non-negative radial solution.
   \end{proof}

   \begin{remark}\label{remark w dash is negative}
     Under the assumption \ref{i1} $\tilde{w}_{k}'>0$ by \cite[Proposition 4.2.2]{PucciSerrinbook2007}.  Since $w_{k}(t)=\tilde{w}_{k}(R_{2}-t)$ and $w(H_{0}(x))=v(x)\lim v_{k}(x)=\lim w_{k}(H_{0}(x))$, we obtain that $w$ is decreasing by taking the limit as $k\to \infty$ in $w_{k}(t_{1})\geq w_{k}(t_{2})$ whenever $t_{1}<t_{2}$.

   Integrating the equation $\tilde{w}_{k}$ from $0$ to $0<t<R_{2}-R_{1}$ and using the fact that $f\geq 0$, one can derive $(R_{2}-t)^{n-1}|\tilde{w}'(t)|^{p-2}\tilde{w}'(t)\geq (R_{2})^{n-1}|\tilde{w}'(0)|^{p-2}\tilde{w}'(0)$, where $\tilde{w}'(0)$ is the right derivative. From this we can infer that $\tilde{w}'(t)\geq 0$ if $\tilde{w}'(0)\geq 0$. Indeed, $\tilde{w}'(0)=\lim_{t^{+}\to 0}\tilde{w}(t)/t\geq 0$ as $\tilde{w}\geq 0$.
\end{remark}
   
   Next, we look at the asymptotic of the solution of \eqref{eqn annulus} as $t\to R_{1}$.
   \begin{proposition}\label{prop annuls asym}
       Assume \ref{a1} and \ref{KO}. Let $v\in W_{loc}^{1,p}(A_{R_{1}}^{R_{2}}(0))$ be the radially symmetric solution of \eqref{eqn annulus}. Then 
       \begin{equation*}
           \limsup_{H_{0}(x)\to R_{1}}\frac{\Psi_{H,p}\left(v(x)\right)}{H_{0}(x)-R_{1}}\leq 1.
       \end{equation*}
   \end{proposition}
\begin{proof}

Let $v(x)=w(H_{0}(x))$. Then $w$ solves \eqref{eqn radial solution}, or equivalently 
\begin{equation*}
    \left(|w'(t)|^{p-2}w'(t)\right)'+\frac{(n-1)}{t}|w'(t)|^{p-2}w'(t)=f(w(t)).
\end{equation*}
Here we follow the proof presented in \cite[Proposition 4.3]{materopaper} with necessary modifications.
 
Since $w'\leq 0$; see Remark \ref{remark w dash is negative}, the above equation can be written as
\begin{equation*}
    \left(-(-w'(t))^{p-1}\right)'-\frac{(n-1)}{t}(-w'(t))^{p-1}=f(w(t)).
\end{equation*}
Which implies
\begin{equation*}
    (p-1)(-w'(t))^{p-2}w''(t)-\frac{(n-1)}{t}(-w'(t))^{p-1}=f(w(t)).
\end{equation*}
    Multiplying by $\frac{p}{p-1}(-w'(t))t^{\frac{p(n-1)}{p-1}}$, we obtain 
    \begin{equation*}
        \left(t^{\frac{p(n-1)}{p-1}}|w'(t)|^{p}\right)'=\frac{p}{p-1}t^{\frac{p(n-1)}{p-1}}w'(t)f(w(t)).
    \end{equation*}
    Let $R_{1}<t<t_{0}<R_{2}$. Integrating the above equation from $t$ to $t_{0}$,
    \begin{equation*}
        t_{0}^{\frac{p(n-1)}{p-1}}|w'(t_{0})|^{p}-t^{\frac{p(n-1)}{p-1}}|w'(t)|^{p}=\frac{p}{p-1}\int_{t}^{t_{0}}s^{\frac{p(n-1)}{p-1}}dF(w(s)).
    \end{equation*}
    By adding and subtracting $\frac{p}{p-1}t^{\frac{p(n-1)}{p-1}}\int_{t}^{t_{0}}dF(w(s))$,
    we obtain
    \begin{equation*}
        \begin{split}
            |w'(t)|^{p}=&\left(\frac{t_{0}}{t}\right)^{\frac{p(n-1)}{p-1}}|w'(t_{0})|^{p}-\frac{p}{p-1}\left\{F(w(t_{0}))-F(w(t))\right\}-\frac{p}{p-1}\int_{t}^{t_{0}}\bigg\{\left(\frac{s}{t}\right)^{\frac{p(n-1)}{p-1}}-1\bigg\}dF(w(s))\\
            &\leq \left(\frac{t_{0}}{t}\right)^{\frac{p(n-1)}{p-1}}|w'(t_{0})|^{p}+\frac{p}{p-1}F(w(t))+\frac{p}{p-1}\int_{t_{0}}^{t}\bigg|\left(\frac{s}{t}\right)^{\frac{p(n-1)}{p-1}}-1\bigg|dF(w(s))\\
            &\leq \left(\frac{t_{0}}{R_{1}}\right)^{\frac{p(n-1)}{p-1}}|w'(t_{0})|^{p}+\frac{p}{p-1}F(w(t))+\frac{p}{p-1}\bigg(\left(\frac{t_{0}}{R_{1}}\right)^{\frac{p(n-1)}{p-1}}-1\bigg)\left\{F(w(t))-F(w(t_{0}))\right\}\\
            &\leq \left(\frac{t_{0}}{R_{1}}\right)^{\frac{p(n-1)}{p-1}}|w'(t_{0})|^{p}+\frac{p}{p-1}F(w(t))+\frac{p}{p-1}\bigg(\left(\frac{t_{0}}{R_{1}}\right)^{\frac{p(n-1)}{p-1}}-1\bigg)F(w(t)).
            \end{split}
    \end{equation*}
    Given any $\epsilon>0$, choose $R_{1}<t_{0}<R_{2}$ such that for all $R_{1}<t\leq t_{0}$
    \begin{equation*}
        \left(\frac{t_{0}}{R_{1}}\right)^{\frac{p(n-1)}{p-1}}|w'(t_{0})|^{p}\leq  \frac{\epsilon p}{2(p-1)}F(w(t)), \quad \text{and}\quad\left(\frac{t_{0}}{R_{1}}\right)^{\frac{p(n-1)}{p-1}}-1<\frac{\epsilon}{2}.
    \end{equation*}
    Here, the first one holds as $F(w(t))\to +\infty$ as $t\to R_{1}$.

    Therefore we have 
    \begin{equation*}
        -w'(t)\leq (1+\epsilon)^{1/p} \left(\frac{p}{p-1}F(w(t))\right)^{1/p}.
    \end{equation*}
    Finally,
    \begin{equation*}
    \begin{split}
         \Psi_{H,p}(w(t))&=\left(\frac{p-1}{p}\right)^{1/p}\int_{w(t)}^{\infty}\frac{ds}{\left\{F(s)\right\}^{1/p}}\\
         &=\left(\frac{p-1}{p}\right)^{1/p}\int^{t}_{R_{1}} \frac{-w'(s)}{\{F(w(s))\}^{1/p}}\,ds\\
         &\leq (1+\epsilon)^{1/p}(t-R_{1}).
         \end{split}
    \end{equation*}
    From this, we obtain
    \begin{equation*}
            \Psi_{H,p}(v(x))=\Psi_{H,p}(w(H_{0}(x)))\leq (1+\epsilon)^{1/p}(H_{0}(x)-R_{1})
    \end{equation*}
    This gives 
    \begin{equation*}
        \limsup_{H_{0}(x)\to R_{1}}\frac{\Psi_{H,p}(v(x))}{{H_{0}}(x)-R_{1}}\leq 1.\qedhere
    \end{equation*}
\end{proof}

  \begin{lemma}\label{lemma asym 2}
       Let $\Omega\subset \rn$ be a bounded domain with $C^{2}$ boundary. Let $u$ be a solution of \eqref{finsler large eqn}. Then  
       \begin{equation*}
           \limsup_{\delta_{H_{0}}(x)\to 0}\frac{\Psi_{H,p}\left(u(x)\right)}{\delta_{H_{0}}(x)}\leq 1.
       \end{equation*}
       
   \end{lemma} 
   \begin{proof}

   Let $x\in \Omega_{\mu}$, then there exists $z(x)\in \partial\Omega$ such that \eqref{rep of x in Omega mu} holds. Denoting $z=z(x)$, let $\mathcal{W}_{R}(x^{ext})$ be the exterior $H_{0}-$ball associated with $z$. 
   
   Consider the solution $v_{r}$ of \eqref{eqn annulus} on $A_{R-r}^{2R}(x^{ext})$ for some $0<r<R$. Denote $\Omega_{r}'=\Omega\cap A_{R-r}^{2R}(x^{ext})$. Both $u$ and $v_{r}$ solve $\Delta_{H}^{p}w=f(w)$ in $\Omega_{r}'$. Since $u\geq 0, v_{r}=0$ on the part of the outer boundary of the annulus in $\Omega$, i.e., $\partial\Omega_{r}'\cap\partial A_{R-r}^{2R}(x^{ext})$, and $v_{r}<\infty$ on the part of the boundary of $\Omega$ in the annulus, i.e., $\partial\Omega_{r}'\cap\partial\Omega$, by the comparison principle $v_{r}\leq u$ on $\Omega'_{r}$. 
   
   Further, for $r_{1}\leq r_{2}$ we have $v_{r_{1}}(x)\geq v_{r_{2}}(x)$ for all $x\in A_{R-r_{1}}^{2R}(x^{ext})$ by the comparison principle. Then $v=\sup v_{r}\leq u$ on $\Omega'=\Omega\cap A^{2R}_{R}(x^{ext})$. Moreover, employing the arguments analogous to the proof of Theorem \ref{main theorem existence}, $v$ is a solution of \eqref{eqn annulus} on $A_{R}^{2R}(x^{ext})$.
    
       Thus, by \eqref{rep of x in Omega mu} and \eqref{rep of center of ball}, since $\Psi_{H,p}$ is decreasing, we obtain for any $x\in \Omega'\cap \Omega_{\mu}$ 

       \begin{equation*}
           \frac{\Psi_{H,p}\left(u(x)\right)}{\delta_{H_{0}}(x)}=\frac{\Psi_{H,p}\left(u(x)\right)}{H_{0}(x-z)}\leq \frac{\Psi_{H,p}\left(v(x)\right)}{H_{0}(x-x^{ext})-R}.
       \end{equation*}
       Now, given $\epsilon>0$, by Proposition \ref{prop annuls asym}, there exists $\delta>0$ such that whenever $H_{0}(x-x^{ext})-R\leq \delta$, one has
       \begin{equation*}
           \frac{\Psi_{H,p}\left(v(x)\right)}{H_{0}(x-x^{ext})-R}< 1+\epsilon.
       \end{equation*}
       Choosing $\delta<\mu$, where $\mu$ is defined before \eqref{rep of x in Omega mu}, we obtain for all $x\in \Omega_{\delta}$ that 
       \begin{equation*}
        \frac{\Psi_{H,p}\left(u(x)\right)}{\delta_{H_{0}}(x)}< 1+\epsilon.
       \end{equation*}
       This completes the proof.
   \end{proof}

Combining the above results, we obtain the following proof of Theorem \ref{dual asym}.
    \begin{proof}[\textbf{Proof of Theorem \ref{dual asym}}]
        Since $\Omega$ is $C^{2}$, applying Lemma \ref{lemma asym 1} and Lemma \ref{lemma asym 2}, we get 
        \begin{equation*}
            1\leq \liminf_{\delta_{H_{0}}(x)\to 0}\frac{\Psi_{H,p}\left(u(x)\right)}{\delta_{H_{0}}(x)}\leq \lim_{\delta_{H_{0}}(x)\to 0}\frac{\Psi_{H,p}\left(u(x)\right)}{\delta_{H_{0}}(x)}\leq  \limsup_{\delta_{H_{0}}(x)\to 0}\frac{\Psi_{H,p}\left(u(x)\right)}{\delta_{H_{0}}(x)}\leq 1.\qedhere
        \end{equation*}
    \end{proof}

    \section{Uniqueness}\label{uniqueness}
    In this section, we assume that $f(t)=t^{q}$ and establish the uniqueness using the asymptotics presented in the previous section.

    Let $c(p,q)$ denote a constant whose value may change from line to line. From Example \ref{example f=tq}, we obtain
    \begin{equation*}
        \Phi_{H,p}(s):=\Psi^{-1}_{H,p}(s)=\frac{c(p,q)}{s^{\frac{p}{q+1-p}}}.
    \end{equation*}
   
    \begin{proof}[\textbf{Proof of Corollary  \ref{thrm asym for f=tq}}]
     Given $\alpha>0$,  by Theorem \ref{dual asym}, there exists $\delta>0$ such that for all $x\in \Omega_{\delta}$ one has 
    \begin{equation*}
        (1-\alpha)\delta_{H_{0}}(x)<\Psi_{H,p}\left(u(x)\right)<(1+\alpha)\delta_{H_{0}}(x).
    \end{equation*}
    Since $\Psi_{H,p}$ is decreasing , so is $\Phi_{H,p}$. This implies 
    \begin{equation*}
        (1+\alpha)^{\frac{-p}{q+1-p}}\Phi(\delta_{H_{0}}(x))<u(x)<(1-\alpha)^{\frac{-p}{q+1-p}}\Phi(\delta_{H_{0}}(x)).
    \end{equation*}
    Rewriting 
    \begin{equation*}
        (1+\alpha)^{\frac{-p}{q+1-p}}<\frac{u(x)}{\Phi(\delta_{H_{0}}(x))}<(1-\alpha)^{\frac{-p}{q+1-p}}.
    \end{equation*}
    Since $1-(1+\alpha)^{\frac{-p}{q+1-p}}\to 0$ and $(1-\alpha)^{\frac{-p}{q+1-p}}-1\to 0$ as $\alpha\to 0$, we infer \eqref{asym f=tq}.
    \end{proof}
   
   \begin{proof}[\textbf{Proof of Theorem \ref{thrm uniqueness}}]
       Let $u$ and $v$ be two solutions of \eqref{finsler large eqn} with $f(t)=t^{q}$, $q>p-1$. Then, 
    \begin{equation*}
        \lim_{\delta_{H_{0}}(x)\to 0}\frac{u(x)}{v(x)}=\lim_{\delta_{H_{0}}(x)\to 0}\frac{u(x)}{\Psi_{H,p}(\delta_{H_{0}}(x))}~\frac{\Psi_{H,p}(\delta_{H_{0}}(x))}{v(x)}=1.
    \end{equation*}
    Therefore, by the second part of Theorem \ref{comparison principle}, $u\leq v$ in $\Omega$. Indeed, $f(\lambda t)\geq\lambda^{p-1}t^{q}$ for all $t>0$ and $\lambda>1$ since $q>p-1$.
    
   We infer uniqueness by interchanging the roles of $u$ and $v$.
   \end{proof} 
    
    \section{Appendix}
This appendix gathers basic existence results about the Finsler p-Laplacian, followed by a brief discussion about the regularity.

Let $\Omega$ be a bounded domain in $\rn$, extend $f$ to be $0$ on $(-\infty,0)$ consider the energy 
 \begin{equation*}
    J(u)=\int_{\Omega}\frac{1}{p}H^{p}\left(\D u\right)+F(u) \,dx, \quad \text{for all}\ u\in W^{1,p}(\Omega),
\end{equation*}
where $F$ is the primitive of $f$ as in \eqref{eqn primitive of f}. Also, consider
\begin{equation}\label{dirchlet problem}
     \begin{cases}
        -\Delta_{H}^{p}u+f(u)=0\quad &\text{in }\Omega\\
        u=g &\text{on } \partial\Omega,
    \end{cases}
\end{equation}
for some $g\in W^{1,p}(\Omega)$. The positivity of $f$ plays a role in the existence of a solution, as given below. 
\begin{definition}\label{dfn weak solution to dirchlet prob}
    A function $u\in W^{1,p}(\Omega)$ is a weak solution of \eqref{dirchlet problem} if
    \begin{equation*}
        -\int_{\Omega}\left<\Fin{u}, \D \phi\right>\,dx=\int_{\Omega}f(u)\phi\,dx \quad \text{for all }\phi\in W_{0}^{1,p}(\Omega)
    \end{equation*}
    and $u-g\in W_{0}^{1,p}(\Omega)$ with $f(u)\in L^{p'}(\Omega)$.
\end{definition}
\begin{remark}
    In the above definition, $u$ also satisfies the weak formulation for all $\Omega'\subset\subset\Omega$ with the test function $\phi\in W^{1,p}_{0}(\Omega')$.
\end{remark}
Clearly, the minimiser of the functional $J$ on $X:=\{u~|~u-g\in W_{0}^{1,p}(\Omega) \text{ and }F(u)\in L^{1}(\Omega)\}$ is the solution of \eqref{dirchlet problem}.Notice that, for our purpose it is sufficient to take $g=c$, a positive constant. Thus, we also assume that $g\in L^{\infty}(\Omega)$ which implies that $X\neq\emptyset$. Next, we show that the minimisers exist:
\begin{lemma}
    Let $g\in W^{1,p}(\Omega)\cap L^{\infty}(\Omega)$. The functional $J$ is coercive, bounded below and weakly lower semi-continuous.
\end{lemma}
\begin{proof}
    By Lemma \ref{lemma_properties of H}, since $F\geq 0$ we infer
    \begin{equation*}
        J(u)\geq \int_{\Omega}\frac{1}{p}H^{p}(\nabla u) \,dx\geq \frac{\theta_{1}^{p}}{p}\int_{\Omega}|\nabla u|^{p}\,dx\geq 0.
    \end{equation*}
    Thus, $J$ is bounded below. The coercivity of $J$ can be obtained by aplying Poincaré inequality to $u-g$. To show that $J$ is weakly lower semi-continuous, let us divide $J$ into two functionals and show that both are weakly lower semi-continuous. Since the functional $u\mapsto \int_{\Omega}H^{p}(\nabla u)\,dx$ is continuous and convex, by \cite[Corollary 3.9, Remark 6]{Brezis2011} it is weakly lower semi-continuous. For the other part, let $v_{n}\rightharpoonup v$ in $W^{1,p}(\Omega)$. If $p< n$, then by the compact Sobolev embedding $v_{n}\to v$ in $L^{p}(\Omega)$ and, if $p\geq n$, then choose $0<\alpha<n^{2}/(p+n)$, then $p<(n-\alpha)^{*}$ and $W^{1,p}(\Omega)\subset W^{1,n-\alpha}(\Omega)$. The compact embedding then gives $v_{n}\to v$ in $L^{p}(\Omega)$. Thus, there is a subsequence $v_{n_{k}}$ such that $v_{n_{k}}\to v$ pointwise a.e. The continuity of $F$ implies $F(v_{n_{k}})\to F(v)$ pointwise  a.e. Finally, Fatou's lemma gives the weakly lower semi-continuity of $u\mapsto \int_{\Omega}F(u)\,dx$.
    \end{proof}
      In view of the above lemma the following holds.
    \begin{proposition}
        The functional $J$ admits a unique minimiser in $X$.
    \end{proposition}
    \begin{proof}
    Let $v_{n}$ be a minimising sequence in $X$. By the weak compactness $v_{n}\rightharpoonup v$ in $W^{1,p}(\Omega)$ up to some subsequence still denoted $v_{n}$. Further, $v$ is a minimiser of $J$ in $W^{1,p}(\Omega)$. The space $W_{0}^{1,p}(\Omega)$ is weakly closed. Thus, $v-g\in W_{0}^{1,p}(\Omega)$. Finally, by Fatou's lemma $F(v)\in L^{1}(\Omega)$ implies $v\in X$. 

    Let $v_{1}$ and $v_{2}$ be two minimisers. 
    By the monotonicity assumption on $f$,
    \begin{equation*}
        \int_{\Omega}\left(f(v_{1})-f(v_{2})\right)(v_{1}-v_{2})\,dx \geq 0.
    \end{equation*}
    By Corollary \ref{strict convexity of H},
    \begin{equation*}
        \int_{\Omega}\left<\Fin{v_{1}}-\Fin{v_{2}}, \nabla (v_{1}-v_{2})\right> \,dx>c \int_{\Omega}H^{p}(\nabla(v_{1}-v_{2}))\,dx,
    \end{equation*}
    if $v_{1}\neq v_{2}$. But since $v_{1}-v_{2}\in W_{0}^{1,p}(\Omega)$,
    \begin{equation*}
        \int_{\Omega}\left<\Fin{v_{1}}-\Fin{v_{2}}, \nabla(v_{1}-v_{2})\right> +\left(f(v_{1})-f(v_{2})\right)(v_{1}-v_{2})=0.
    \end{equation*}
    Combining all three inequalities, we conclude that $v_{1}=v_{2}$.
    \end{proof}
    As we have extended $f$ to whole $\mathbb{R}$, $v$ need not be non-negative. To remedy this, take $g\geq 0$ and consider
    \begin{equation*}
        J(v^{+})=\int_{\Omega}\frac{1}{p}H^{p}(\nabla v^{+})+F(v)\,dx=\int_{\Omega}\frac{1}{p}H^{p}(\nabla v)+F(v)\,dx-\int_{\Omega} \frac{1}{p}H^{p}(\nabla v^{-})\,dx\leq J(v).
    \end{equation*}
    Also, $v^{+}=g$ a.e. on the boundary. We can obtain non-negative solutions.
    
  Finally, we state the regularity result.
  \begin{proposition}\label{prop-regularity}
      If $u$ be a bounded solution of \eqref{dirchlet problem}, then $u\in C_{loc}^{1,\alpha}(\Omega)\cap C(\overline{\Omega})$ for some $\alpha\in (0,1)$.
  \end{proposition}
  \begin{proof}
      The strong convexity of $H$ ensures the structural conditions like ellipticity and growth of the operator. Thus, by \cite[Proposition 3.1]{CozziFarinaValdinoci2016}, \cite[Proposition 3.1]{CozziFarinaValdinoci2014} (see the proof of Proposition 3.1 \cite{CozziFarinaValdinoci2014}), and \cite{Tolksdorf1984}, the solutions are in $C_{loc}^{1,\alpha}(\Omega)$ for some $\alpha\in (0,1)$. Moreover, when $g=c$, a constant, the solution is in $C(\overline{\Omega})$ by \cite{Lieberman1988}. 

      Although the result in latter reference is stronger, for our purpose the continuity up to the boundary is sufficient.  
  \end{proof}

    \section{Acknowledgements}
    The author acknowledges the financial support from the Prime Minister's Research Fellowship (PMRF), Government of India (NO. 2302262), during the first part of the research, and through the ANRF PMECRG (ANRF/ ECRG/ 2024/ 004743/ PMS) grant held by Dr. Indranil Chowdhury (Indian Institute of Technology Kanpur) for the second part.

    The author thanks Dr. Indranil Chowdhury for the support and valuable suggestions.
    \vspace{-0.1cm}
\renewcommand\refname{Bibliography}
\bibliographystyle{amsplain}
\bibliography{ref.bib}
\end{document}